\documentclass{amsart}
\usepackage{amssymb, graphics}
\usepackage[all]{xy}

 \usepackage[applemac]{inputenc}
\usepackage[cyr]{aeguill}
\usepackage[francais]{babel}

\usepackage[margin = 1in]{geometry}

\newtheorem{lemme}{Lemme}[section]
\newtheorem{proposition}[lemme]{Proposition}
\newtheorem{corollaire}[lemme]{Corollaire}
\newtheorem{theoreme}[lemme]{Th\'eor\`eme}
\newtheorem{Theoreme}{Théorème}

\newtheorem{exemple}[lemme]{Exemple}
\newtheorem{definition}[lemme]{D\'efinition}
\newtheorem{remarque}[lemme]{Remarque}

\newcommand\cf{cf\@. }

\newcommand\eg{e\@.g\@. }

\newcommand\pa{ \partial}

\newcommand\bbB{\mathbb B}
\newcommand\bbC{\mathbb C}

\newcommand\bbN{\mathbb N}

\newcommand\bbS{\mathbb S}

\renewcommand\Re{\operatorname{Re}}
\renewcommand\Im{\operatorname{Im}}

\newcommand\CI{\mathcal{C}^{\infty}}

\newcommand\Diff{\operatorname{Diff}}
\newcommand\Ric{\operatorname{Ric}}

\newcommand\cC{\mathcal{C}}
\newcommand\cA{\mathcal{A}}

\newcommand\db{\overline{\pa}}

\newcommand\cV{\mathcal{V}}
\newcommand\cU{\mathcal{U}}

\newcommand\End{\operatorname{End}}

\newcommand\pr{\operatorname{pr}}

\newcommand\phg{\operatorname{phg}}

\newcommand\Id{\operatorname{Id}}

\newcommand\hu{\widehat{u}}
\newcommand\homega{\widehat{\omega}}
\newcommand\hDelta{\widehat{\Delta}}

\newcommand\CR{\operatorname{CR}}
\newcommand\ahc{\operatorname{AHC}}

\begin{document}
\title[]
{Polyhomog\'en\'eit\'e des m\'etriques asymptotiquement hyperboliques complexes le long du flot de Ricci}

\author{Fr\'ed\'eric Rochon}
\begin{abstract}   
On montre que la polyhomogénéité à l'infini d'une métrique asymptotiquement hyperbolique complexe est préservée par le flot de Ricci-DeTurck.  De plus, si la métrique initiale est «lisse jusqu'au bord», alors il en sera de même pour les solutions du flot de Ricci normalisé et du flot de Ricci-DeTurck.  Lorsque la métrique initiale est aussi kählérienne,  des résultats plus précis sont obtenus en termes d'un potentiel.    
\end{abstract}
\maketitle

\tableofcontents

\section*{Introduction}  \label{int.0}

L'étude du flot de Ricci sur les variétés riemanniennes complètes non-compactes a débuté avec les travaux de Shi \cite{Shi1989}, qui a montré que lorsque la courbure de la métrique initiale est uniformément bornée, il existe au moins pour un court laps de temps une solution $g(t)$ complète de courbure uniformément bornée.  Comme l'ont montré Chen et Zhu \cite{Chen-Zhu2006}, la solution  est unique lorsqu'on requiert que la courbure soit uniformément bornée.  Au cours des dernières années, il y a eu plusieurs travaux sur l'existence d'une solution pour tout temps positif et sur la convergence du flot vers une métrique d'Einstein, entre autres pour des surfaces complètes \cite{Ji-Mazzeo-Sesum2009, AAR2013, IMS2010, Giesen-Topping2011} et plus généralement pour les variétés kählériennes \cite{Chau2004, Lott-Zhang2011}.  Plusieurs résultats de stabilité ont aussi été obtenus, notamment pour l'espace euclidien \cite{SSS2008}, pour l'espace hyperbolique réel \cite{SSS2011} ou certaines de ses déformations \cite{Bahuaud2011, QSW2013}, pour des variétés hyperboliques de volume fini \cite{Bamler2010a} et pour des espaces symétriques de type non-compact \cite{Bamler2010b}.    

Dans la plupart de ces travaux, on se restreint à un type particulier de géométrie à l'infini, que ce soit un cône, un cusp ou l'espace hyperbolique.  Pour faire une étude spectrale du laplacien associé à de telles métriques, il est souvent utile d'introduire une compactification de la variété complète par une variété à bord ou à coins, voir entre autres \cite{Mazzeo-Melrose, EMM, MelroseAPS, MelroseGST, MazzeoEdge, Mazzeo-MelrosePhi, Vaillant}.  En fait, pour obtenir un prolongement méromorphe de la résolvante du laplacien, il faut typiquement non seulement savoir que la métrique considérée est approximée par le modèle géométrique à l'infini, mais encore qu'elle admette un développement asymptotique polyhomogène complet à l'infini, autrement dit au bord de la compactification, une sorte de «régularité à l'infini».  Lorsqu'une telle métrique est modifiée par le flot de Ricci, il est donc naturel de se demander si cette régularité à l'infini sera préservée par le flot. 

Dans \cite{AAR2013}, une réponse positive à cette question a été obtenue pour des surfaces munies de métriques complètes approximées par des cusps ou des entonnoirs à l'infini, ce qui a permis d'étudier le déterminant régularisé du Laplacien le long flot et de montrer que dans une classe conforme donnée et en imposant une condition de normalisation, le déterminant est maximal pour la métrique hyperbolique lorsque celle-ci existe.    Pour les métriques asymptotiquement hyperboliques, la persistance le long du flot de Ricci d'un développement asymptotique lisse a été obtenue par Bahuaud \cite{Bahuaud2011}.  Pour des métriques kählériennes complètes de volume fini sur des variétés quasi-projectives, la persistance d'un développement asymptotique lisse a été obtenue dans \cite{Rochon-Zhang2012} en introduisant une certaine compactifaction par une variété à coins.

Dans le présent article, nous obtenons un tel résultat pour des métriques asymptotiquement hyperboliques complexes, abrégé métriques $\ahc$.  Rappelons que si $X$ est une variété compacte à bord de dimension paire $2n$ et que $c:\pa X\times [0,\epsilon)\hookrightarrow X$ est un choix de voisinage tubulaire du bord, alors une métrique $g$ sur $X\setminus \pa X$ est asymptotiquement hyperbolique complexe s'il existe une structure $\CR$ $(\pa X,H,J)$ strictement pseudoconvexe le long d'une distribution de contact $H$ et une forme de contact $\eta$ telles que asymptotiquement près du bord (voir la Définition~\ref{ahc.2} ci-bas pour plus de détails), on ait
\begin{equation}
  c^*g \sim \frac{4d\rho^2}{\rho^2}+ \frac{\eta^2}{\rho^4}+ \frac{d\eta(\cdot,J\cdot)}{\rho^2},
\label{int.1}\end{equation}
où $\rho$ est la coordonnée sur le deuxième facteur de $\pa X\times [0,\epsilon)$.
En prenant $X=\bbB^{2n}\subset \bbC^n$ la boule fermée de rayon $1$, $\rho = \sqrt{1-|z|^2}$ et $\eta=i(\pa-\db)(1-|z|^2)$ la forme de contact standard sur $\bbS^{2n-1}$, le terme de droite de \eqref{int.1} est précisément la métrique hyperbolique complexe près du bord, normalisée de sorte que sa courbure sectionnelle holomorphe soit égale à $-1$.  Plus généralement, comme expliqué dans \cite{Biquard2000}, le comportement à l'infini d'une métrique $\ahc$ est modelé sur celui de la métrique hyperbolique complexe, ce qui signifie en particulier qu'une métrique $\ahc$ est asymptotiquement Einstein avec constante de Einstein $\lambda= -\frac{n+1}{2}$.  Pour les métriques $\ahc$, il est donc naturel de considérer le flot de Ricci normalisé
\begin{equation}
         \frac{\pa g}{\pa t}= -\Ric(g)- \frac{n+1}{2}g, \quad g(0)=g_0.
\label{int.2}\end{equation}   
Comme cette équation n'est pas parabolique, il est souvent utile de considérer plutôt le flot de Ricci-DeTurck donné en coordonnées locales par
\begin{equation}
  \frac{\pa g_{ij}}{\pa t}=  -\Ric(g)_{ij}- \frac{n+1}{2} g_{ij} + \nabla_i W_j + \nabla_j W_i, \quad g(0)=g_0,
\label{int.3}\end{equation}
avec $t\to W(t)$ la famille de champs de vecteurs donnée par
\begin{equation}
              W^k= \frac12 g^{pq}\left( \Gamma^k_{pq}(t)-\Gamma^k_{pq}(0) \right),
\label{int.4}\end{equation}
où $\Gamma^k_{pq}(t)$ est le symbole de Christoffel de la métrique $g(t)$.  Le résultat principal de cet article est le suivant, voir le Théorème~\ref{ahc.7} pour l'énoncé détaillé.
\begin{Theoreme}
Si $g_0$ est une métrique $\ahc$ polyhomogène, alors la solution $g(t)$ au flot de Ricci-DeTurck \eqref{int.3} est aussi une métrique $\ahc$ polyhomogène.  De même, $g(t)$ sera une métrique $\ahc$ lisse jusqu'au bord si $g_0$ est une métrique $\ahc$ lisse jusqu'au bord.   
\label{int.5}\end{Theoreme}   
Lorsque $g_0$ est une métrique lisse jusqu'au bord, on peut en particulier utiliser ce théorème pour conclure que la solution $g(t)$ au flot de Ricci normalisé \eqref{int.2} est une métrique $\ahc$ lisse jusqu'au bord, voir le Corollaire~\ref{ahc.8} ci-bas.  Pour démontrer le Théorème~\ref{int.5}, l'idée est de construire formellement par récurrence le développement asymptotique polyhomogène de $g(t)$ à l'aide de celui de $g_0$ pour ensuite établir à l'aide d'une estimation de décroissance pour les solutions d'équations paraboliques linéaires (le Corollaire~\ref{par.3}) que ce développement asymptotique est bien celui de $g(t)$.  L'approche utilisée est assez robuste et peut servir dans d'autres cadres, notamment pour donner une preuve différente du résultat de Bahuaud \cite{Bahuaud2011} sur le développement asymptotique des métriques conformément compactes.  

Lorsque la métrique $\ahc$ initiale $g_0$ est aussi kählérienne et que $\omega_0$ est sa forme de Kähler, on peut remplacer le flot de Ricci normalisé par une équation parabolique pour un potentiel $u$,
\begin{equation}
  \frac{\pa u}{\pa t}= \log\left( \frac{(\omega_t+ i\pa \db u)^n}{\omega_0^n} \right) -\frac{n+1}{2}u, \quad u(0,\cdot)=0.
\label{int.6}\end{equation}
où
$$
     \omega_t= -\frac{2}{n+1}Ric(\omega_0)+ e^{-\frac{n+1}{2}t}(\omega_0+\frac{2}{n+1}\Ric(\omega_0))
$$
et $\omega_t+i\pa \db u$ est la forme kählérienne de la solution $g(t)$ au flot de Ricci normalisé \eqref{int.2}.  Dans ce cas, en suivant une approche similaire à celle développée dans \cite{Rochon-Zhang2012}, on obtient le résultat suivant, voir le Théorème~\ref{ricci.5} ci-bas pour l'énoncé détaillé.  
\begin{Theoreme}
Si $g_0$ est une métrique kählérienne $\ahc$ polyhomogène, alors la solution $u(t)$ de \eqref{int.6} est polyhomogène.  En particulier, la solution $g(t)$ au flot de Ricci normalisé est une métrique $\ahc$ polyhomogène.  
\label{int.7}\end{Theoreme}   

Lorsque le flot existe pour tout temps et converge, insistons sur le fait que les Théorèmes~\ref{int.5} et \ref{int.7} ne font aucune assertion sur la polyhomogénéité de la métrique vers laquelle le flot converge.  En fait, à la limite, la nature du développement asymptotique peut changer substantiellement.  Par exemple, sur un domaine strictement pseudoconvexe, on peut commencer avec une métrique kählérienne $\ahc$ lisse jusqu'au bord pour laquelle le flot de Ricci converge vers la métrique de Cheng-Yau.  Or, par les résultats de Lee et Melrose \cite{Lee-Melrose}, on sait que la métrique de Cheng-Yau admet un développement polyhomogène, mais qu'elle n'est typiquement pas lisse jusqu'au bord.    

L'article est organisé comme suit.  Le \S~\ref{Theta.0} consiste en un bref rappel sur la notion de $\Theta$-structure introduite par Epstein, Mendoza et Melrose \cite{EMM}.  On obtient ensuite un résultat de décroissance pour certaines solutions d'équations linéaires $\Theta$-paraboliques dans le \S~\ref{depar.0}, ce qui nous permet d'étudier la polyhomogénéité des solutions de certaines de ces équations. Dans le \S~\ref{ahc.0}, on donne la démonstration du résultat principal de l'article, à savoir que le flot de Ricci-DeTurck préserve la polyhomogénéité à l'infini pour les métriques $\ahc$.  Enfin, on obtient dans le \S~\ref{kahc.0} un résultat similaire en termes d'un potentiel pour les métriques $\ahc$ kählériennes.

\section{Rappel sur la notion de $\Theta$-structure} \label{Theta.0}

La notion de $\Theta$-structure a été introduite par Epstein, Melrose et Mendoza \cite{EMM} pour construire la résolvante du Laplacien associé à la métrique de Bergman et la métrique de Cheng-Yau \cite{Cheng-Yau} sur un domaine strictement pseudoconvexe.  Dans ce cadre, la $\Theta$-structure est donnée par la structure CR sur le bord.  

\begin{definition}
Une \textbf{$\Theta$-structure} sur une variété à bord $X$ (compacte et de classe $\CI$) est donnée par une $1$-forme $\Theta\in\CI(\pa X;T^*X)$ sur le bord de $X$  telle que $\iota^*\Theta$ ne s'annule nulle part sur $\pa X$, où $\iota:\pa X\hookrightarrow X$ est l'inclusion canonique.     
\label{tm.1}\end{definition}

Soit $\rho \in \CI(X)$ une fonction de définition du bord, c'est-à-dire que $\rho\ge 0$, $\pa X= \rho^{-1}(0)$ et $d\rho$ ne s'annule pas sur $\pa X$.  Une $\Theta$-structure définit alors une algèbre de Lie de champs de vecteurs donnée par
\begin{equation}
\cV_{\Theta}(X)= \{ \xi\in \CI(X;TX) \quad | \quad  \left. \xi \rho\right|_{\pa X}=0, \quad \overline{\Theta}(\xi)\in \rho^2 \CI(X) \},
\label{tm.2}\end{equation}
où la forme $\overline{\Theta}\in \CI(X;TX)$ telle que $\left. \overline{\Theta} \right|_{\pa X}=\Theta$ est un choix d'extension de $\Theta$ à $X$.  En particulier, remarquons que les champs de vecteurs de $\cV_{\Theta}(X)$ sont tous tangents au bord.  Comme on peut aisément le vérifier (voir Proposition~1.6 dans \cite{EMM}), l'algèbre de Lie $\cV_{\Theta}(X)$ ne dépend pas du choix de l'extension $\overline{\Theta}$ et du choix de la fonction de définition du bord $\rho$.  De plus, elle demeure inchangée lorsque définie à partir d'un autre représentant de la classe conforme de $\Theta$.  Elle est donc canoniquement associée à la $\Theta$-structure.  On peut décrire l'algèbre de Lie $\cV_{\Theta}$ localement près d'un point $p\in \pa X$ en choisissant une base locale de champs de vecteurs $N,T, Y_1,\ldots, Y_{n-2}$,   définis dans un voisinage $\cU$ de $p$ de sorte que     
\begin{equation}
\begin{aligned}
  d\rho(N)&=1, \; \overline{\Theta}(N)=0, \\
    d\rho(Y_i)&=0, \;\overline{\Theta}(Y_i)=0, \; \; i=1,\ldots,n-2,  \quad n=\dim X,\\
  d\rho(T)&=0, \; \overline{\Theta}(T)=1.
\end{aligned}  
\label{tm.3}\end{equation}
Près de $p$, un élément $\xi\in\cV_{\Theta}(X)$ est alors de la forme 
\begin{equation}
  \xi= a \rho N + \sum_{i=1}^{n-2} b_i \rho Y_i + c\rho^2 T, \quad a,b_1,\ldots, b_{n-2},c\in \CI(X). 
\label{tm.4}\end{equation}

En utilisant ces bases locales, on peut donc définir un fibré vectoriel ${}^{\Theta}TX$ sur $X$ tel que 
\begin{equation}
\cV_{\Theta}(X)=(\iota_{\Theta})_*\CI(X;{}^{\Theta}TX), 
\label{tm.5}\end{equation}
où $\iota_{\Theta}: {}^{\Theta}TX\to TX$ est un morphisme canonique de fibrés vectoriels qui donne lieu à un isomorphisme lorsque restreint à $X\setminus \pa X$ et qui est défini par
$\iota_{\Theta}(v)=0$ pour $v\in \left. {}^{\Theta}TX\right|_{\pa X}$. 

\begin{remarque}
Le fibré ${}^{\Theta}TX$ est isomorphe à $TX$, mais l'application canonique $\iota_{\Theta}$ ne donne pas un isomorphisme entre ces deux fibrés.  Le fibré ${}^{\Theta}TX$ possède en fait une structure d'algébroïde de Lie avec $(\iota_{\Theta})_*$ comme ancrage (anchor map en anglais).  
\end{remarque}

\begin{definition}
Une \textbf{$\Theta$-métrique lisse jusqu'au bord} est une métrique riemannienne $g$ sur l'intérieur de $X$ de la forme 
$$
            g= (\iota^{-1}_{\Theta})^*\left. g_{\Theta}\right|_{X\setminus \pa X},
$$
où $g_{\Theta}\in \CI(X; {}^{\Theta}T^*X\otimes {}^{\Theta}T^*X)$ est une métrique euclidienne pour le fibré vectoriel ${}^{\Theta}TX$.  
\label{tm.6}\end{definition}

Si $d\rho$, $\alpha_1,\ldots,\alpha_{n-2}, \overline{\Theta}\in \CI(\cU,T^*X)$ forment la base locale duale à la base de champs de vecteurs décrite en \eqref{tm.3}, alors 
$\frac{d\rho}{\rho}, \frac{\alpha_1}{\rho}, \ldots, \frac{\alpha_{n-2}}{\rho}, \frac{\overline{\Theta}}{\rho^2} \in \CI(\cU;{}^{\Theta}T^*X)$ est une base locale pour le dual du fibré
${}^{\Theta}TX$.  Dans un tel voisinage, une $\Theta$-métrique lisse jusqu'au bord est donc de la forme
\begin{multline}
      g=  a \frac{d\rho^2}{\rho^2} + \sum_{i,j}\frac{g_{ij}}{\rho^2} \alpha_i\otimes \alpha_j + b\frac{\overline{\Theta}^2}{\rho^4} + \sum_i c_i \left( \frac{d\rho}{\rho}\otimes \frac{\alpha_i}{\rho} + \frac{\alpha_i}{\rho}\otimes \frac{d\rho}{\rho} \right)  \\ + \sum_i d_i \left( \frac{\overline{\Theta}}{\rho^2}\otimes \frac{\alpha_i}{\rho} + \frac{\alpha_i}{\rho}\otimes \frac{\overline{\Theta}}{\rho^2} \right) + e\left(  \frac{\overline{\Theta}}{\rho^2}\otimes \frac{d\rho}{\rho} + \frac{d\rho}{\rho}\otimes \frac{\overline{\Theta}}{\rho^2}\right), 
\label{tm.7}\end{multline}
avec $a,g_{ij},b,c_i,d_i, e\in \CI(\cU)$.  

Comme on peut le vérifier sans trop de difficulté, une $\Theta$-métrique lisse jusqu'au bord est complète de volume infini.  De plus, si $g_1$ et $g_2$ sont deux $\Theta$-métriques lisses jusqu'au bord associées à la même $\Theta$-structure, alors celles-ci sont quasi-isométriques, c'est-à-dire qu'il existe une constante positive $C$ telle que
$$
          \frac{g_1}{C}\le g_2 \le Cg_1.
$$

La notion de $\Theta$-métrique lisse jusqu'au bord est un peu stricte et il est parfois nécessaire de considérer une classe de métriques un peu plus grande.
\begin{definition}
Une \textbf{$\Theta$-métrique} est une métrique riemannienne sur $X\setminus \pa X$ qui est quasi-isométrique à une $\Theta$-métrique lisse jusqu'au bord.     

\end{definition}
Par exemple, sur un domaine strictement pseudoconvexe, la métrique de Cheng-Yau \cite{Cheng-Yau} est quasi-isométrique à une $\Theta$-métrique lisse jusqu'au bord.  Cependant, comme la solution formelle de Fefferman\cite{Fefferman76} l'indique,  la métrique de Cheng-Yau  n'est typiquement pas une $\Theta$-métrique lisse jusqu'au bord.  Les travaux de Lee et Melrose \cite{Lee-Melrose} montrent toutefois qu'elle admet un développement asymptotique polyhomogène près du bord.  Pour décrire précisément ce qu'on entend par développement polyhomogène, rappelons d'abord brièvement la notion d'ensemble indiciel (index set en anglais) \cite{MelroseAPS}.

\begin{definition}
Un ensemble indiciel $E$ est un sous-ensemble discret de $\bbC\times \bbN_0$ tel que 
\begin{gather}
(z_j,k_j)\in E, \quad |(z_j,k_j)|\to \infty \; \Longrightarrow \; \Re z_j \to \infty, \\
\label{tm.8b} (z,k)\in E \; \Longrightarrow (z+p,k)\in E \; \forall \ p\in \bbN,  \\
(z,k)\in E \; \Longrightarrow (z,p)\in E \; \forall \ p\in \{0,\ldots,k\}.
\end{gather}
\label{tm.8}\end{definition}  
À un ensemble indiciel $E$ donné correspond l'espace $\cA^E_{\phg}(X)$ des fonctions $f\in \CI(X\setminus \pa X)$ ayant un développement asymptotique près de $\pa X$ de la forme
\begin{equation}
 f\sim \sum_{(z,k)\in E} a_{(z,k)} \rho^z (\log \rho)^k, \quad a_{(z,k)}\in \CI(X),
\label{tm.9}\end{equation} 
où le symbole $\sim$ signifie que pour tout $N\in\bbN$, on a
\begin{equation}
  f-\underset{\Re z\le N}{\sum_{(z,k)\in E}} a_{(z,k)} \rho^z (\log \rho)^k \in \dot{\cC}^N(X),
\label{tm.10}\end{equation}
où $\dot{\cC}^N(X)$ est l'espace des fonctions $N$ fois différentiables sur $X$ s'annulant sur $\pa X$ ainsi que toutes leurs dérivées jusqu'à l'ordre $N$.  Par la propriété \eqref{tm.8b}, l'espace $\cA^E_{\phg}(X)$ est automatiquement un $\CI(X)$-module.  Si $V$ est un fibré vectoriel sur $X$, on peut donc aussi définir l'espace des sections polyhomogènes de $V$ par rapport à un ensemble indiciel $E$ par
$$
                    \cA^E_{\phg}(X;V)= \cA^E_{\phg}(X) \otimes_{\CI(X)} \CI(X;V).
$$

Remarquons que les coefficients $a_{(z,k)}$ dans le développement asymptotique \eqref{tm.9} dépendent du choix de la fonction de définition du bord $\rho$, mais que l'espace $\cA^E_{\phg}(X)$ reste le même si on change de fonction de définition du bord.    L'espace des fonctions lisses $\CI(X)$ est un cas particulier d'un espace de fonctions polyhomogènes, puisque qu'il correspond à l'espace $\cA^E_{\phg}(X)$ avec $E=\bbN_0\times \{0\}$ comme ensemble indiciel.  En prenant $E=\emptyset$, on obtient plutôt l'espace $\cA^E_{\phg}(X)= \dot{\cC}^{\infty}(X)$  des fonctions lisses sur $X$ s'annulant avec toutes leurs dérivées sur le bord $\pa X$.    

Pour les $\Theta$-métriques, on considérera un type particulier d'ensemble indiciel.

\begin{definition} On dira qu'un ensemble indiciel $E$ est \textbf{positif} si $\bbN_0\times \{0\}\subset E$ et si 
\begin{gather}
(z,k)\in E \: \Longrightarrow \; \Im z=0 \; \mbox{et} \;  \Re z\ge 0,\\
        (0,k)\in E   \; \Longrightarrow \; k=0.      
\end{gather}     
\label{tm.11a}\end{definition}

\begin{definition}
Une \textbf{$\Theta$-métrique polyhomogène} est une $\Theta$-métrique $g$ sur $X\setminus \pa X$ de la forme 
$$
g= (\iota^{-1}_{\Theta})^*\left. g_{\Theta}\right|_{X\setminus \pa X}
$$
pour une métrique euclidienne $g_{\Theta}\in \cA^E_{\phg}(X; {}^{\Theta}T^*X\otimes {}^{\Theta}T^*X)$ avec $E$ un ensemble indiciel positif.  
\label{tm.11}\end{definition}

Outre l'espace des fonctions lisses $\CI(X)$ et celui des fonctions polyhomogènes par rapport à un ensemble indiciel $E$, on peut introduire des espaces fonctionnels adaptés à la $\Theta$-structure.  Pour ce faire, fixons une $\Theta$-métrique $g$ et soit $\nabla$ sa connexion de Levi-Civita.  Cela induit une métrique euclidienne et une connexion euclidienne (qu'on dénotera aussi $g$ et $\nabla$) sur le fibré vectoriel
$$
       T^r_s(X\setminus \pa X)= \underbrace{T(X\setminus \pa X)\otimes \cdots \otimes T(X\setminus\pa X)}_{r \; \mbox{fois}}\otimes  \underbrace{T^*(X\setminus \pa X)\otimes \cdots \otimes T^*(X\setminus\pa X)}_{s \; \mbox{fois}}
$$  
des $(r,s)$-tenseurs.  Plus généralement, si $V\to X\setminus \pa X$ est un fibré vectoriel euclidien lisse muni d'une connexion euclidienne, alors le fibré vectoriel euclidien
$T^r_sX\otimes V$ a une connexion euclidienne induite que l'on dénotera à nouveau $\nabla$.  Pour $k\in \bbN_0$, on dénotera alors par $\cC_{\Theta}^k(X\setminus\pa X;V)$ l'espace des sections continues $\sigma$ de $V$ telles que 
\begin{equation}
   \nabla^j \sigma \in \cC^0(X\setminus\pa X;T^0_j(X\setminus \pa X) \otimes V) \quad \mbox{et} \quad \sup_{p\in X\setminus \pa X}| \nabla^j \sigma(p) |_g < \infty,  \quad \forall \ j\in \{0,\ldots,k\},
\label{tm.13}\end{equation}
où $|\cdot |_g$ est la norme définie par la métrique euclidienne sur $T^r_s(X\setminus\pa X) \otimes V$.  L'espace $\cC^k_{\Theta}(X;V)$ est un espace de Banach ayant pour norme 
\begin{equation}
        \| \sigma \|_{k}= \sum_{j=0}^k \sup_{p\in X\setminus \pa X} |\nabla^j \sigma(p)|_g.
\label{tm.14}\end{equation}
En prenant l'intersection sur tous les $k\in \bbN$, on obtient un espace de Fréchet
\begin{equation}
    \cC^{\infty}_{\Theta}(X\setminus \pa X;V) = \bigcap_{k\in \bbN}  \cC^k_{\Theta}(X\setminus \pa X;V)
\label{tm.15}\end{equation}
ayant pour semi-normes $\|\cdot\|_{k}$ pour $k\in \bbN$.  Dans la base locale \eqref{tm.3}, on a que la restriction sur $\cU\cap (X\setminus \pa X)$ d'une section $\sigma\in \CI_{\Theta}(X;V)$ est telle que
$$
       \sup_{\cU\cap(X\setminus\pa X)}  | (\rho N)^i (\rho^2 T)^j \rho^{|\nu|}Y^{\nu} \sigma  | <\infty,  \quad \forall \ i,j\in\bbN_0, \;  \forall \nu\in \bbN_0^{n-2}.
$$
\begin{exemple}
Si $\rho\in \CI(X)$ est une fonction de définition du bord telle que $\rho(p)<1$ pour tout $p\in X$, alors pour $k\in\bbN$, la fonction $(\log \rho)^{-k}$ n'est pas une fonction lisse sur $X$, mais  elle est un élément de $\CI_{\Theta}(X\setminus \pa X)$.
\label{tm.15b}\end{exemple}

 Pour $\alpha\in (0,1]$,  il est aussi utile de définir l'espace de Hölder  $\cC^{0,\alpha}_{\Theta}(X\setminus \pa X;V)$ comme l'espace des sections continues $\sigma$ de $V$ telles que
 \begin{equation}
     \| \sigma\|_{0,\alpha}:= \|\sigma\|_{0}+ \sup   \left\{ \frac{|  P_{\gamma}(\sigma(\gamma(0))) - \sigma(\gamma(1))|_g}{  \ell(\gamma)^\alpha} \quad | \quad \gamma\in \CI([0,1];X\setminus \pa X), \; \gamma(0)\neq \gamma(1)\right\}<\infty,
 \label{tm.16}\end{equation}
où $P_{\gamma}: \left. V\right|_{\gamma(0)} \to \left. V\right|_{\gamma(1)}$  est le transport parallèle le long de $\gamma$ et $\ell(\gamma)$ est la longueur de $\gamma$ par rapport à la métrique $g$.  On peut aussi définir l'espace de Hölder $\cC^{k,\alpha}_{\Theta}(X\setminus \pa X;V)$ comme étant donné par les sections $\sigma\in \cC^k_{\Theta}(X\setminus \pa X; V)$ telles que $\nabla^k\sigma\in \cC^{0,\alpha}_{\Theta}(X\setminus \pa X; T^0_k(X\setminus \pa X)\otimes V)$.  L'espace de Hölder $\cC^{k,\alpha}_{\Theta}(X\setminus \pa X;V)$ est un espace de Banach avec norme donnée par
$$
        \| \sigma \|_{k,\alpha}:=  \| \sigma \|_{k-1} + \| \nabla^k\sigma \|_{0,\alpha}.
$$

On peut aussi introduire une version parabolique de ces espaces.  Dénotons aussi par $V$ l'image inverse de $V$ par la projection $[0,T]\times X\to X$.  On définit alors 
$\cC^{k}_{\Theta}([0,T]\times X\setminus \pa X;V)$ comme étant l'espace des sections $\sigma$ de $V$ telles que pour tous $i,j \in \bbN_0$ avec $2i+j\le k$, on ait
$$
 \pa_t^i \nabla^j  \sigma \in \cC^0([0,T]\times X\setminus \pa X; T^0_j(X\setminus \pa X)\otimes V) \quad \mbox{et} \quad \sup_{t\in [0,T]} \sup_{p\in X\setminus \pa X} | \pa_t^i \nabla^j\sigma(t,p)|_g < \infty.
$$
C'est un espace de Banach avec norme donnée par
$$
  \| \sigma\|_{k}:= \sum_{2i+j\le k}\sup_{t\in [0,T]} \sup_{p\in X\setminus \pa X} | \pa_t^i \nabla^j\sigma(t,p)|_g.
$$ 
On dénote par
$$
      \CI_{\Theta}([0,T]\times X\setminus \pa X;V) = \bigcap_{k=0}^{\infty}\cC^k_{\Theta}([0,T]\times X\setminus \pa X;V)
$$
l'espace de Fréchet associé.  
Plus généralement, pour $\alpha\in (0,1]$, on peut considérer l'espace de Hölder parabolique $\cC^{0,\alpha}_{\Theta}([0,T]\times X\setminus \pa X; V)$ constitué des sections continues
$\sigma\in \cC^0([0,T]\times X\setminus \pa X; V)$ telles que 
\begin{multline}
 \| \sigma\|_{0,\alpha}:= \|\sigma\|_{0}+ \sup_{t\in[0,T]} \sup \left\{ \frac{|  P_{\gamma}(\sigma(\gamma(0))) - \sigma(\gamma(1))|_g}{  \ell(\gamma)^\alpha}\quad | \quad \gamma\in \CI([0,1];\{t\}\times X\setminus \pa X), \; \gamma(0)\neq \gamma(1) \right\} \\+ \sup_{x\in X\setminus\pa X}  \sup_{t\neq t'} \frac{|\sigma(t,x)-\sigma(t',x)|}{|t-t'|^{\frac{\alpha}{2}}} <\infty.
 \end{multline}
Enfin, pour $k\in\bbN$ et $\alpha\in (0,1]$, on définit l'espace de Hölder parabolique $\cC^{k,\alpha}_{\Theta}([0,T]\times X\setminus \pa X ; V)$ comme étant l'espaces des sections
$\sigma\in \cC^{k}_{\Theta}([0,T]\times X\setminus \pa X;V)$ telles que
$$
  \| \sigma \|_{k,\alpha}:= \| \sigma \|_{k-1} + \|\nabla^k\sigma\|_{0,\alpha} < \infty.
$$
C'est un espace de Banach avec norme donnée par $\|\cdot \|_{k,\alpha}$.

\section{Décroissance de solutions d'équations $\Theta$-paraboliques} \label{depar.0}

À partir de l'algèbre de Lie $\cV_{\Theta}(X)$, on peut aussi définir l'espace des $\Theta$-opérateurs différentiels.  

\begin{definition}
L'espace $\Diff^k_{\Theta,\infty}(X)$ des $\Theta$-opérateurs différentiels lisses jusqu'au bord agissant sur $\CI(X)$ est engendré par $\CI(X)$ et la composition d'au plus $k$ 
champs de vecteurs de $\cV_{\Theta}(X)$.  Plus généralement, si $E$ est un ensemble indiciel, on définit alors par
$$
              \Diff^k_{\Theta,E}(X)= \cA^E_{\phg}(X) \otimes_{\CI(X)}  \Diff^k_{\Theta,\infty}(X)
$$
l'ensemble des $\Theta$-opérateurs différentiels qui sont polyhomogènes par rapport à $E$. 
\label{tm.12}\end{definition}

Un exemple naturel de $\Theta$-opérateur différentiel lisse jusqu'au bord est donné par le laplacien associé à une $\Theta$-métrique lisse jusqu'au bord.  Si la $\Theta$-métrique est polyhomogène par rapport à un ensemble indiciel $E$, alors son laplacien sera un $\Theta$-opérateur différentiel polyhomogène par rapport à un ensemble indiciel contenant $E$, mais possiblement un peu plus grand.  Comme nous aurons affaire à des $\Theta$-métriques qui ne sont pas nécessairement lisses jusqu'au bord ou polyhomogènes, il nous faudra considérer une classe d'opérateurs un peu plus grande que les $\Theta$-opérateurs différentiels polyhomogènes.  
\begin{definition}
Soit $V\to X$ un fibré vectoriel euclidien muni d'une connexion euclidienne.  On dénote par $\Diff^k_{\Theta}(X\setminus \pa X;V)$ l'espace des opérateurs différentiels P d'ordre $k$ de la forme
$$
         P\sigma= \sum_{j=0}^k a_j \cdot \nabla^j \sigma,  \quad a_j\in \CI_{\Theta}(X\setminus \pa X; T^j_0(X\setminus \pa X)\otimes \End(V)), \quad \sigma\in \CI_{\Theta}(X\setminus \pa X; V),
 $$ 
 où $a_j\cdot \nabla^j \sigma$ est vu comme un élément de $\CI_{\Theta}(X\setminus\pa X;V)$ et où $\nabla$ est la connexion induite par la connexion euclidienne de $V$ et la connexion de Levi-Civita d'une $\Theta$-métrique lisse jusqu'au bord fixée.   Le \textbf{symbole principal} d'un opérateur $P\in \Diff^k_{\Theta}(X\setminus \pa X;V)$, dénoté $\sigma_k(P)$, est $i^k a_k$, où $a_k$ est le coefficient du terme d'ordre $k$.  Comme la notation le suggère, $\sigma_k(P)$ ne dépend pas du choix de la connexion utilisée pour décrire l'opérateur $P$.       
\label{tm.13}\end{definition}

Remarquons que $\Diff^k_{\Theta,\infty}(X;V)$ est un sous-espace de $\Diff^k_{\Theta}(X\setminus \pa X;V)$ donné par les opérateurs de $\Diff^k_{\Theta}(X\setminus \pa X;V)$
dont les coefficients $a_j$ sont en fait des éléments de $\CI(X;{}^{\Theta}T^j_0(X)\otimes V)$.

\begin{definition}
Un opérateur $L\in \Diff^2_{\Theta}(X\setminus \pa X;V)$ est \textbf{uniformément $\Theta$-elliptique} si le symbole principal $\sigma_2(L)\in \CI(X\setminus \pa X; T(X\setminus \pa X) \otimes T(X\setminus \pa X)\otimes V)$ est tel que pour tout $x\in X\setminus \pa X$, 
$$
      \sigma_2(L)(\xi,\xi)=  -g^{ij}\xi_i\xi_j \quad \forall \ \xi\in T^*_x(X\setminus \pa X)\setminus \{0\},
$$
pour une certaine $\Theta$-metric $g$.  Autrement dit, $L$ est uniformément $\Theta$-elliptique si son symbole principal est le même que celui d'un laplacien associé à une $\Theta$-métrique.  De même, si $[0,T] \ni t\to L_t\in  \Diff^2_{\Theta}(X\setminus \pa X;V)$ est une famille lisse d'opérateurs uniformément $\Theta$-elliptiques,  on dira que l'opérateur $\frac{\pa}{\pa t}- L_t$ est \textbf{uniformément $\Theta$-parabolique}.
\label{tm.14}\end{definition}

\begin{proposition}
Soit $V$ un fibré vectoriel euclidien sur $X$.  Pour $t\in [0,T]$, soit $t\mapsto L_t\in \Diff^2_{\Theta}(X\setminus \pa X;V)$ une famille lisse d'opérateurs telle que l'opérateur $\pa_t - L_t$ soit uniformément $\Theta$-parabolique.  Alors pour $u_0\in \cC^{k+2,\alpha}_{\Theta}(X\setminus \pa X;V)$ et $f\in\cC^{k,\alpha}_{\Theta}([0,T]\times (X\setminus \pa X;V))$ , l'équation
$$
             \frac{\pa u}{\pa t} - L_t u =f, \quad \left. u\right|_{t=0}= u_0,
$$
possède une solution $u\in\cC^{k+2,\alpha}_{\Theta}([0,T]\times X\setminus \pa X;V)$.  De plus, il existe une constante $C>0$ dépendant de la famille $L_t$ telle que
$$
          \| u\|_{k+2,\alpha}\le C\left(\|u_0\|_{k+2,\alpha}+ \| f\|_{k,\alpha}\right).
$$
\label{par.2}\end{proposition}
\begin{proof}
C'est un résultat standard.  Rappelons brièvement sa démonstration.  D'abord, en considérant $u-u_0$ au lieu de $u$, on peut toujours se ramener au cas où $u_0=0$.  Soit $\{\Omega_j\}_{j\in \bbN}$ une suite d'ouverts relativement compacts  avec bord lisse recouvrant $X\setminus \pa X$ et tels que $\Omega_j\subset \Omega_{j+1}$ pour tout $j\in\bbN$.  Pour chacun de ces ouverts, soit $\psi_j \in \CI_c(\Omega_{j+1})$ une fonction telle que $\psi_j\equiv 1$ sur $\Omega_j$ et $\sup_{\Omega_{j+1}} |\psi_j| =1$.  Pour chaque $j$, on sait alors par le Théorème~7.1 dans \cite[Chapiter~7, \S7]{LSU1968}  que l'équation
$$
         \frac{\pa \phi_j}{\pa t}= L_t \phi_j + \psi_j f, \quad \phi_j(0,\cdot)=0, \; \left. \phi_j\right|_{\pa \Omega_{j+1}\times[0,T]} =0,
$$   
possède une unique solution $\phi_j\in \cC^{k+2,\alpha}([0,T]\times \Omega_{j+1};V)$.  
En mettant  $L_t$ sous la forme
 $$
 L_t= \sum_{k=0}^2 a_k(t)\cdot \nabla^k, 
$$
où au temps $t$, $\nabla$ est la connexion induite par un choix de connexion euclidienne pour $V$ indépendante du temps et la connexion de Levi-Civita pour la $\Theta$-metric $g(t)$  induite par $\sigma_{2}(L_t)$, alors en fait
\begin{equation}
   L_t v= \Delta_{g(t)}v + a_1\nabla v + a_0 v.
\label{max.3}\end{equation}
Pour $j,\ell\in\bbN$, on voit donc que sur $\Omega_{\ell+1}$,
\begin{equation}
      \frac{\pa}{\pa t} (\psi_{\ell}\phi_j) = L_t(\psi_{\ell} \phi_j) - (\Delta_{g(t)}\psi_{\ell})\phi_j -2\langle\nabla\psi_{\ell}, \nabla\phi_j\rangle_{g(t)} -a_1\cdot (\nabla\psi_{\ell})\phi_j + \psi_{\ell}\psi_j f, \quad \left. \psi_{\ell}\phi_j\right|_{\pa \Omega_{\ell+1}\times[0,T]} =0.
\label{max.2}\end{equation}
Par la discussion dans \cite[p. 320 et p.584]{LSU1968}, on déduit de \eqref{max.2}  des estimations indépendantes de $j$ sur les normes $\cC^{k+2,\alpha}_{\Theta}$ des $\psi_{\ell}\phi_j$.  Par le théorème d'Arzela-Ascoli, il existe donc une sous-suite $\{\phi_{j_q}\}_{q\in\bbN}$ et une solution $u\in \cC^{k+2,\alpha}_{\Theta}(X\setminus \pa X;V)$ telle que $\phi_{j_q}\to u$ dans $\cC^{k+2,\alpha}_{\Theta}([0,T]\times X\setminus \pa X;V)$ uniformément sur des sous-ensembles compacts de $X\setminus \pa X$.   
Pour obtenir l'estimation pour $\| u\|_{k+2,\alpha}$, notons que, comme une $\Theta$-métrique est un exemple de géométrie bornée (bounded geometry en anglais) au sens de Cheng-Yau \cite{Cheng-Yau}, \cf \cite[p.171]{Lee-Melrose}, on peut obtenir l'estimation voulue à partir de l'estimation de Schauder parabolique locale.  

\end{proof}

Ce dernier résultat se généralise sans trop de difficulté aux espaces de Hölder paraboliques à poids, ce qui permet du même coup de donner une preuve assez facile de l'unicité de la solution.  
\begin{corollaire}
Soit $V$ un fibré vectoriel euclidien sur $X$.  Pour $t\in [0,T]$, soit $t\mapsto L_t\in \Diff^2_{\Theta}(X\setminus \pa X;V)$ une famille lisse d'opérateurs telle que l'opérateur $\pa_t - L_t$ soit uniformément $\Theta$-parabolique.  Soit $\rho\in\CI(X)$ une fonction de définition du bord.  Alors pour $u_0\in \rho^{\beta}(\log \rho)^{\ell}\cC^{k+2,\alpha}_{\Theta}(X\setminus \pa X;V)$ et $f\in\rho^{\beta}(\log \rho)^{\ell}\cC^{k,\alpha}_{\Theta}([0,T]\times (X\setminus \pa X);V)$  avec $\beta> 0$ et $\ell\in\bbN_0$ (ou $\beta=\ell=0$), l'équation
$$
             \frac{\pa u}{\pa t} - L_t u =f, \quad \left. u\right|_{t=0}= u_0,
$$
possède une solution unique $u\in \rho^\beta(\log \rho)^{\ell}\cC^{k+2,\alpha}_{\Theta}([0,T]\times X\setminus \pa X;V)$.  De plus, il existe une constante $C>0$ dépendant de la famille $L_t$ telle que
\begin{equation}
          \left\| \frac{u}{\rho^{\beta}(\log \rho)^{\ell}}\right\|_{k+2,\alpha}\le C\left(  \left\| \frac{u_0}{\rho^{\beta}(\log \rho)^{\ell}}\right\|_{k,\alpha}  + \left\| \frac{f}{\rho^{\beta}(\log \rho)^{\ell}}\right\|_{k,\alpha}\right).
\label{par.3a}\end{equation}
\label{par.3}\end{corollaire}
\begin{proof}
Par l'Exemple~\ref{tm.15b}, on voit que la famille d'opérateurs conjugués 
$$
       \widetilde{L}_t:= \rho^{-\beta}(\log \rho)^{-\ell}\circ L_t \circ \rho^{\beta} (\log \rho)^{\ell}
$$
est aussi une famille lisse d'opérateurs dans $\Diff^2_{\Theta}(X\setminus \pa X)$.  De plus, on voit aisément que $\sigma_2(\widetilde{L}_t)= \sigma_2(L_t)$, de sorte que 
$\pa_t -\widetilde{L}_t$  est aussi uniformément $\Theta$-parabolique.  En posant, 
$$
         \widetilde{u}= \frac{u}{\rho^{\beta}(\log \rho)^{\ell}}, \quad \widetilde{u_0}= \frac{u_0}{\rho^{\beta}(\log \rho)^{\ell}}, \quad \widetilde{f}= \frac{f}{\rho^{\beta}(\log \rho)^{\ell}},
$$
on se ramène donc à l'équation
$$
   \frac{\pa \widetilde{u}}{\pa t}- \widetilde{L}_t \widetilde{u}= \widetilde{f}, \quad \left. \widetilde{u}\right|_{t=0}= \widetilde{u}_0
$$
 et l'existence de la solution découle de la proposition précédente.

Pour ce qui est de l'unicité de la solution, supposons que $u_1$ et $u_2$ soient deux solutions.  Il faut montrer que $u_1=u_2$ .    D'abord, en remplaçant $u_1$, $u_2$, $L_t$ et $f$ par $\rho^{\beta}u_1$, $\rho^{\beta}u_2$, $\widetilde{L}_t=\rho^{\beta} L_t \rho^{-\beta}$ et $\rho^{\beta}f$ avec $\beta>0$ suffisamment grand, on peut toujours se ramener au cas où $|u_1|$ et $|u_2|$ décroissent vers 0 à l'infini.    En posant $u= u_1-u_2$, on voit alors que 
 $$
     \pa_t u= L_t u, \quad \quad \left. u\right|_{t=0}= 0.
$$
En mettant $L_t$ sous la forme \eqref{max.3}, un calcul simple montre alors que
\begin{equation}
\begin{aligned}
     \frac{\pa}{\pa t} |u|^2 &= \Delta_{g(t)}|u|^2 -2|\nabla u|^2_{g(t)} + 2\langle u, a_1\cdot \nabla u\rangle +2 a_0 |u|^2   \\
            &= \Delta_{g(t)}|u|^2 -2|\nabla u|^2_{g(t)} + 2\langle a_1^*\cdot u,  \nabla u\rangle_{g(t)} +2 a_0 |u|^2   \\ 
            &\le \Delta_{g(t)}|u|^2 -2|\nabla u|^2_{g(t)} +  |a_1^*\cdot u|^2_{g(t)} +  |\nabla u|^2_{g(t)} +2 a_0 |u|^2  \\
            &\le     \Delta_{g(t)}|u|^2 + C|u|^2       
            \end{aligned}     
\label{max.1}\end{equation}
pour une constante $C>0$ dépendant de $a_0$ et $a_1^*$. 
Puisque $|u|$ décroît vers 0 à l'infini, on peut appliquer le principe du maximum à cette équation pour conclure que $u\equiv 0$, donc que $u_1=u_2$. 
 
\end{proof}

On peut utiliser ces résultats pour montrer que les solutions de certaines équations uniformément $\Theta$-paraboliques sont lisses jusqu'au bord ou polyhomogènes.  Pour ce faire, il faut considérer la somme et l'union d'ensembles indiciels.  
Si $E$ et $F$ sont deux ensembles indiciels, on définit leur somme par
\begin{equation}
    E+F= \{ (z_1+z_2,k_1+k_2)\in\bbC \; | \; (z_1,k_1)\in E, \; (z_2,k_2)\in F\}.
\end{equation}
On vérifie aisément que $E+F$ est un ensemble indiciel.  Avec cette opération de sommation, les ensembles indiciels constituent un semi-groupe commutatif.   On vérifie de même que l'union $E\cup F$ est un ensemble indiciel.  
L'intérêt de ces opérations sur les ensembles indiciels vient du fait que 
$$
  f\in \cA^E_{\phg}(X), \; g\in\cA^F_{\phg}(X)\quad \Longrightarrow \quad f+g\in \cA_{\phg}^{E\cup F}(X) \; \mbox{et} \; fg \in \cA^{E+F}_{\phg}(X).
$$ 
   
Si $E$ est un ensemble indiciel positif (voir la Définition~\ref{tm.11a}), on peut lui associer l'ensemble indiciel
$$
                E_{\infty}= \sum_{j=1}^{\infty} E = \bigcup_{n=1}^{\infty} \sum_{j=1}^{n} E.
$$
Par construction, cet ensemble indiciel est tel que $E_\infty +E= E_{\infty}$.  Remarquons enfin que si $E$ et $F$ sont des ensembles indiciels positifs, alors $E\cup F\subset E+F$. 

\begin{theoreme}
Soit $V\to X$ un fibré vectoriel euclidien.  Pour $t\in [0,T]$ et $E$ un ensemble indiciel positif, soit $t\mapsto L_t\in \Diff^2_{\Theta,E}(X;V)$ une famille lisse d'opérateurs telle que l'opérateur $\pa_t - L_t$ soit uniformément $\Theta$-parabolique.  Alors pour $u_0\in \cA^{F}_{\phg}(X;V)$ et $f\in\cA^{G}_{\phg}([0,T]\times X;V)$ avec $F$ et $G$ des ensembles indiciels positifs, l'équation
$$
             \frac{\pa u}{\pa t} - L_t u =f, \quad \left. u\right|_{t=0}= u_0
$$
possède une solution unique  $u\in \cA^{H}_{\phg}([0,T]\times X; V)$, où $H$ est l'ensemble indiciel donné par 
$$
           H= E_\infty + (F \cup G).
$$
\label{par.4}\end{theoreme} 

\begin{proof}
Par le Corollaire~\ref{par.3}, on sait que l'équation possède une solution unique $u\in\cC^{\infty}_{\Theta}([0,T]\times X\setminus \pa X;V)$.  Il nous reste donc à montrer que cette solution est en fait contenue dans $\cA^{H}_{\phg}([0,T]\times X;V)$.  D'abord, en remplaçant $u$ par $u-u_0$ si nécessaire, on peut se ramener au cas $u_0=0$, ce que nous supposerons.  Soit $\{z_j\}_{j\in\bbN_0}$ une suite strictement croissante de nombres réels positifs ou nuls avec
\begin{gather}
        z_0=0 \; \mbox{et} \;  (z_j,0)\in H \;\forall j\in \bbN, \\
         (z,k)\in H\; \Longrightarrow \; z=z_j \; \mbox{pour un certain} \; j\in\bbN_0.  
\end{gather}
Pour montrer que $u$ est bien un élément de $\cA^H_{\phg}([0,T]\times X;V)$, il suffit donc de montrer que pour tout $j\in\bbN_0$, il existe $u_j\in \cA^H_{\phg}([0,T]\times X;V)$ tel que 
$u_j(0,\cdot)\equiv 0$ et
\begin{equation}
      u-u_j \in \rho^{\nu}\cC^{\infty}_{\Theta}([0,T]\times X\setminus \pa X;V)  \quad \forall \; \nu < z_{j+1}.   
\label{par.5}\end{equation}
Nous allons procéder par induction sur $j$ pour établir ce résultat.  Supposons donc que pour $j\in\bbN_0$, on a déjà trouvé une fonction $u_{j-1}\in \cA^H_{\phg}([0,T]\times X;V)$ avec 
$u_{j-1}(0,\cdot)\equiv 0$ satisfaisant \eqref{par.5}.  Si $j=0$, il est sous-entendu ici qu'on prend alors $u_{-1}=0$, de sorte que $u-u_{-1}\in \cC^{\infty}_{\Theta}([0,T]\times X\setminus \pa X;V)$.

Posons alors $v= u-u_{j-1}$, de sorte que
\begin{equation}
  \frac{\pa v}{\pa t}= L_t v +f^j,  \quad f^j:= f+L_t u_{j-1} -\frac{\pa u_{j-1}}{\pa t}.
\label{par.6}\end{equation}
Puisque $H+E=H$, on a automatiquement que $f^j \in \cA^{H}_{\phg}([0,T]\times X;V)$.  Étant donné que $v\in \rho^{\nu}\CI_{\Theta}([0,T]\times X\setminus \pa X;V)$ pour $\nu<z_j$, on aura aussi que   $f^j\in  \rho^{\nu}\CI_{\Theta}([0,T]\times X\setminus \pa X;V)$ pour $\nu<z_j$.  Cela implique qu'il existe $k\in\bbN_0$ tel que 
$$
f^j\in \rho^{z_j}(\log \rho)^{k}\CI_{\Theta}([0,T]\times X\setminus \pa X;V).
$$
Pour trouver un candidat pour le coefficient du terme d'ordre $\rho^{z_j}(\log \rho)^k$ dans le développement polyhomogène de $v$, on peut prétendre un instant que $v$ est bien polyhomogène avec $v\in\rho^{z_j}(\log \rho)^k\CI_{\Theta}([0,T]\times X\setminus \pa X;V)$ et restreindre l'équation \eqref{par.6} au coefficient d'ordre $\rho^{z_j}(\log \rho^k)$ sur le bord, ce qui donne l'équation d'évolution
\begin{equation}
   \frac{\pa v_{\pa}}{\pa t}= \ell_t v_{\pa} + f^j_{\pa}, \quad v_{\pa}(0,\cdot)\equiv 0,  \quad \ell_t\in\CI([0,T]\times \pa X;\End(V)) , \quad f^j_{\pa} \in \CI([0,T]\times \pa X;V). 
\label{par.7}\end{equation} 
Précisément, si dans la base locale de champs de vecteurs \eqref{tm.3}, on a 
$$
    L_t= \sum_{p,q} a^{pq} \xi_p \xi_q + \sum_p b^p \xi_p +c, \quad \mbox{avec} \; \xi_p = \rho Y_p, p=1,\ldots, n-2, \quad \xi_{n-1}= \rho T \; \mbox{et} \; \xi_n= \rho N,
$$
alors dans l'ouvert $\cU$ où cette base locale est définie, on a 
$$
      \ell_t= \left. \left( a^{nn} z_j^2 + b^n z_j +c  \right)\right|_{\pa X}.
$$
L'équation d'évolution \eqref{par.7} est une équation différentielle ordinaire linéaire d'ordre 1 qui se résout aisément.  Soit $v_{\pa}\in \CI([0,T]\times \pa X;V)$ sa solution.  Soit 
$\Xi: \CI(\pa X;V) \to \CI(X;V)$ une application d'extension, c'est-à-dire que
\begin{equation}
     \left.\Xi(h)\right|_{\pa X}=h \quad \forall \; h\in \CI(\pa X;V).
\label{ext.1}\end{equation}
 Posons alors $w= v- w_{\pa}$ avec $w_{\pa}= \rho^{z_j}(\log \rho)^k \Xi(v_{\pa})$, de sorte que $w$ satisfait à l'équation d'évolution
 $$
   \frac{\pa w}{\pa t} = L_t w + h^j, \quad w(0,\cdot)\equiv 0, \quad \mbox{avec} \quad h^j= f^j + L_t w_{\pa} - \frac{\pa w_{\pa}}{\pa t}. 
 $$
 Comme $f^j$ et $w_{\pa}$ sont des éléments de $\cA^H_{\phg}([0,T]\times X;V)$ et $\rho^{z_j}(\log \rho)^k \CI_{\Theta}([0,T]\times X;V)$, on aura aussi que $h^j$ est un élément de ces deux espaces.  Par définition de $w_\pa$, on voit aisément que le coefficient d'ordre $\rho^{z_j}(\log \rho)^k$ de $h^j$ est en fait nul.  
 
On voit donc que $h^j \in \rho^{z_j}(\log \rho)^{k-1}\CI_{\Theta}([0,T]\times X;V)$.  Par le Corollaire~\ref{par.3}, on a donc que 
 $$
          w\in \rho^{z_j}(\log \rho)^{k-1} \CI_{\Theta}([0,T]\times X\setminus \pa X;V).
 $$  
 En répétant cet argument $k$ fois, on peut donc trouver $b_k, b_{k-1}, \ldots, b_1, b_0\in \CI([0,T]\times X;V)$ avec $b_i(0,\cdot)\equiv 0$ tels que 
 $$
         \hat{v}= v- \sum_{\ell=0}^k  \Xi(b_{\ell}) \rho^{z_j}(\log \rho )^{\ell}    
 $$
 a pour équation d'évolution
 $$
                   \frac{\pa \hat{v}}{\pa t} = L_t \hat{v} + \hat{f}^j, \quad \hat{v}(0,\cdot)\equiv 0, \quad \mbox{avec} \quad \hat{f}^j \in \cA^H_{\phg}([0,T]\times X)\cap \rho^{z_j}\CI_{\Theta}([0,T]\times X\setminus \pa X)
 $$
 et $\left. \rho^{-z_j}\hat{f}^j\right|_{\pa X}=0$.  Cela veut dire que $\hat{f}^j\in \rho^\nu\CI_{\Theta}([0,T]\times X\setminus \pa X)$ pour $\nu < z_{j+1}$, donc par le Corollaire~\ref{par.3}, que 
 $\hat{v}\in \rho^\nu\CI_{\Theta}([0,T]\times X\setminus \pa X)$ pour $\nu < z_{j+1}$.  Il suffit donc de prendre  
 $$
         u_{j}= u_{j-1} + \sum_{\ell=0}^k  \Xi(b_{\ell}) \rho^{z_j}(\log \rho )^{\ell}
 $$
 dans \eqref{par.5}, ce qui termine la démonstration.
\end{proof}

En prenant $E=F=G=\bbN_0$ dans le théorème précédent, on obtient le cas particulier suivant.  

\begin{corollaire}
Pour $t\in [0,T]$, soit $t\mapsto L_t\in \Diff^2_{\Theta,\infty}(X;V)$ une famille lisse d'opérateurs telle que l'opérateur $\pa_t - L_t$ soit uniformément $\Theta$-parabolique.  Alors pour $u_0\in \CI(X;V)$ et $f\in\CI([0,T]\times X;V)$, l'équation
$$
             \frac{\pa u}{\pa t} - L_t u =f, \quad \left. u\right|_{t=0}= u_0
$$
possède une solution unique  $u\in \CI([0,T]\times X;V)$. 
\label{par.8}\end{corollaire}

\section{Polyhomogénéité des métriques $\ahc$ le long du flot de Ricci-DeTurck} \label{ahc.0}

Soit $X$ une variété compacte à bord de dimension $2n$.  Soit $\rho$ une fonction de définition du bord et $c: \pa X \times [0,\epsilon)\to X$ un voisinage tubulaire du bord tel que 
$$
                 \rho\circ c(x,t)=t.
$$ 
Supposons que le bord $\pa X$ soit muni d'une structure $\CR$ $(\pa X,H,J)$ définie le long d'une distribution de contact $H$ ayant $\eta \in \Omega^1(\pa X)$ comme forme de contact.  Supposons de plus que la structure $\CR$ soit strictement pseudoconvexe, c'est-à-dire que $d\eta(\cdot, J\cdot)$ définit une métrique euclidienne sur la distribution $H$.  Dans ce contexte, on a une $\Theta$-structure associée donnée par
\begin{equation}
  \Theta= \left.\overline{\Theta}\right|_{\pa X}, \quad \overline{\Theta}= c_* \pr_1^* \eta,  
\label{ahc.1}\end{equation}
où $\pr_1: \pa X\times [0,\epsilon)\to \pa X$ est la projection sur le premier facteur.

\begin{definition}
Une \textbf{métrique asymptotiquement hyperbolique complexe} (abrégé métrique $\ahc$) g sur $X\setminus \pa X$ est une $\Theta$-métrique pour la $\Theta$-structure \eqref{ahc.1} telle qu'il existe une 1-forme $\widetilde{\eta}\in \Omega^1(\pa X)$ dans la classe conforme de $\eta$ et $\delta>0$ tels que près du bord
$$
            g- c_*\left( \frac{4 d\rho^2}{\rho^2} + \frac{(\pr_1^*\widetilde{\eta})^2}{\rho^4} + \frac{\pr_1^*d\widetilde{\eta}(\cdot, J\cdot)}{\rho^2} \right) \in \rho^{\delta}\CI_{\Theta}(X\setminus\pa X; T^0_2(X\setminus\pa X)).
$$ 
On dira qu'elle est polyhomogène ou lisse jusqu'au bord si elle est polyhomogène ou lisse jusqu'au bord en tant que $\Theta$-métrique.  
\label{ahc.2}\end{definition}

La métrique hyperbolique complexe est évidemment  un exemple de métrique $\ahc$.   Comme le nom le suggère, le comportement à l'infini d'une métrique $\ahc$ est en fait modelé sur celui de la métrique hyperbolique complexe \cite[I.1.B]{Biquard2000}.  En particulier, on déduit de \cite[I.1.B, équation (1.7)]{Biquard2000} qu'une métrique $\ahc$ est asymptotiquement Einstein au sens suivant.
\begin{lemme}
Soit $g$ une métrique $\ahc$ sur $X\setminus \pa X$.  Alors il existe $\delta>0$ tel que
$$
              \Ric(g) + \frac{n+1}{2} g \in \rho^{\delta}\CI_{\Theta}(X\setminus\pa X; T^0_2(X\setminus \pa X)).
$$
\label{ahc.3}\end{lemme}
Lorsqu'une telle métrique est Einstein (\eg la métrique hyperbolique complexe), sa constante d'Einstein est donc nécessairement $\lambda=-\frac{n+1}{2}$.  Cela suggère de considérer le flot de Ricci normalisé 
\begin{equation}
    \frac{\pa g}{\pa t} = -\Ric(g) -\frac{n+1}{2}g, \quad g(0)=g_0,
\label{ahc.4}\end{equation}
avec $g_0$ une métrique $\ahc$.
Comme la métrique $g_0$ est complète et sa courbure sectionnelle est bornée, on sait par les travaux de Shi \cite{Shi1989} que le flot possède une solution lisse
existant au moins jusqu'à un certain temps $T>0$ et qu'il existe une contante $C>0$ telle que
$$
                           \frac{g_0}{C} \le g(t) \le Cg_0  \quad \forall \; t\in [0,T].
$$
Grâce aux travaux de Chen et Zhu \cite{Chen-Zhu2006}, on sait aussi que cette solution est unique.  Par les estimations de Shi \cite[\S1, (4)]{Shi1989} pour les dérivées du tenseur de courbure, on sait enfin que $g(t)\in \CI_{\Theta}([\delta,T)\times X\setminus \pa X; T^0_2(X\setminus \pa X))$ pour $\delta>0$.  Ces estimations se détériorent toutefois lorsque $\delta$ approche de $0$ du fait que les constantes apparaissant dans ces estimations ne tiennent compte que de la norme $\cC^0$ du tenseur de courbure.  L'avantage de cette approche est que l'on voit directement que la solution au flot existera tant que la norme $\cC^0$ du tenseur de  courbure reste bornée, ce qui donne une caractérisation simple du temps maximal d'existence du flot.  Pour prendre la limite $\delta\to 0$ et voir que la solution est bien dans  $\CI_{\Theta}([0,T)\times X\setminus \pa X; T^0_2(X\setminus \pa X))$, il faut obtenir des estimations qui prennent en compte les dérivées du tenseur de courbure de la métrique initiale.  La manière la plus simple est sans doute de se ramener à une équation parabolique quasi-linéaire et d'appliquer un argument de point fixe.

  En effet, le flot de Ricci n'est pas une équation parabolique, mais par un truc remontant à DeTurck \cite{Deturk}, on peut se ramener à une équation parabolique quasi-linéaire en modifiant la solution par une famille de difféomorphismes.  Plus précisément, le flot de Ricci-DeTurck normalisé est donné en coordonnées locales par
\begin{equation}
  \frac{\pa g_{ij}}{\pa t}=  -\Ric(g)_{ij}- \frac{n+1}{2} g_{ij} + \nabla_i W_j + \nabla_j W_i, \quad g(0)=g_0,
\label{ahc.5}\end{equation}
avec $t\to W(t)$ la famille de champs de vecteurs donnée par
\begin{equation}
              W^k= \frac12 g^{pq}\left( \Gamma^k_{pq}(t)-\Gamma^k_{pq}(0) \right),
\label{ahc.4a}\end{equation}
où $\Gamma^k_{pq}(t)$ est le symbole de Christoffel de la métrique $g(t)$.  Soit $t\to \varphi_t$ avec $\varphi_0=\Id$ la famille de difféomorphismes engendrée par la famille de champs de vecteurs $W$.  Alors si $g(t)$ est une solution de \eqref{ahc.5}, un calcul simple, voir par exemple \cite[p.81]{Chow-Knopf}, montre que $(\varphi_t^{-1})^*g(t)$ est une solution du flot de Ricci normalisé \eqref{ahc.4}.   Si on utilise des «$\sim$» pour dénoter les connexions de Levi-Civita et les tenseurs de courbure associés à la métrique initiale $g_0$, alors un calcul standard, voir par exemple \cite[Lemma~2.1]{Shi1989}, montre que le flot de Ricci-DeTurck peut être mis explicitement sous la forme d'une équation parabolique quasi-linéaire         
\begin{equation}
\begin{aligned}
\frac{\pa g_{ij}}{\pa t} &= \frac{g^{ab}}{2}\widetilde{\nabla}_a\widetilde{\nabla}_b g_{ij} -\frac{n+1}{2} g_{ij} -\frac{g^{ab}\widetilde{g}^{pq}}{2}\left(g_{ip}\widetilde{R}_{jaqb} +g_{jp}\widetilde{R}_{iaqb}\right)  \\
  & \hspace{0.5cm}+ \frac{g^{ab}g^{pq}}{4}\left(   \widetilde{\nabla}_ig_{pa}\widetilde{\nabla}_j g_{qb} + 2\widetilde{\nabla}_a g_{jp} \widetilde{\nabla}_{q}g_{ib} 
   -2\widetilde{\nabla}_a g_{jp}\widetilde{\nabla}_b g_{iq} - 2\widetilde{\nabla}_j g_{pa} \widetilde{\nabla}_b g_{iq} -2\widetilde{\nabla}_i g_{pa} \widetilde{\nabla}_b g_{jq}\right), \\
  g(0)&= g_0.
\end{aligned}
\label{ahc.6}\end{equation}
À l'aide des estimations de Schauder \eqref{par.3a} du Corollaire~\ref{par.3}, on peut alors appliquer un argument de point fixe et un argument de régularité parabolique (voir par exemple \cite{Bahuaud2011} pour le cas des métriques asymptotiquement hyperboliques) pour conclure que \eqref{ahc.6} possède une solution unique $g\in \CI_{\Theta}([0,T)\times X\setminus \pa X; T^0_2(X\setminus \pa X))$ pour au moins un court laps de temps $T>0$.  En combinant avec les résultats de Shi \cite{Shi1989} déjà mentionnés, on sait en fait que cette solution existera aussi longtemps que la norme $\cC^0$ du tenseur de courbure reste sous contrôle.

Comme notre choix de normalisation le suggère, le flot de Ricci-DeTurck normalisé préserve le comportement asymptotiquement hyperbolique complexe d'une métrique.

\begin{lemme}
Soit $g_0$ une métrique $\ahc$ sur $X$ et soit $t\to g(t)$, $t\in [0,T)$ la solution du flot de Ricci-DeTurck normalisé ayant $g_0$ comme métrique initiale.  Alors il existe $\delta>0$ tel que
$$
           g(t)-g_0\in \rho^{\delta}\CI_{\Theta}(X\setminus \pa X;T^0_2(X\setminus \pa X))
$$
pour tout $t\in [0,T)$.  En particulier, $g(t)$ est aussi une métrique $\ahc$ pour les mêmes choix de structure $\CR$, de fonction de définition du bord $\rho$ et de voisinage tubulaire $c: \pa X\times [0,\epsilon)\to X$. 
\label{ahc.6a}\end{lemme}
\begin{proof}
Posons $u=g-g_0$.  L'équation d'évolution de $u$  peut être mise sous la forme
\begin{equation}
  \frac{\pa u_{ij}}{\pa t} = \frac{g^{ab}}{2}\widetilde{\nabla}_a \widetilde{\nabla}_b u_{ij} + (Q_1(g_0,u, \widetilde{\nabla}u)\cdot \widetilde{\nabla}u)_{ij} + (Q_2(g_0,u,\widetilde{\nabla}u)\cdot u)_{ij} + f_{ij}, \quad u(0)=0, 
\label{ahc.9}\end{equation}  
où $Q_1$ est à valeur dans $\CI_{\Theta}([0,T)\times X\setminus \pa X; T^3_2 (X\setminus \pa X))$, $Q_2$ est à valeur dans  $\CI_{\Theta}([0,T)\times X\setminus \pa X; T^2_2 (X\setminus \pa X))$, la notation $\cdot$ indique une contraction naturel d'indices et $f\in \CI_{\Theta}([0,T)\times X\setminus \pa X; T^0_2 (X\setminus \pa X))$ est donnée par
\begin{equation}
\begin{aligned}
f_{ij}&= \frac{\widetilde{g}^{ab}}{2}\widetilde{\nabla}_a \widetilde{\nabla}_b \widetilde{g}_{ij} -\frac{n+1}{2} \widetilde{g}_{ij} -\frac{\widetilde{g}^{ab}\widetilde{g}^{pq}}{2}\left(\widetilde{g}_{ip}\widetilde{R}_{jaqb} +\widetilde{g}_{jp}\widetilde{R}_{iaqb}\right)  \\
  & \hspace{0.5cm}+ \frac{\widetilde{g}^{ab}\widetilde{g}^{pq}}{4}\left(   \widetilde{\nabla}_i \widetilde{g}_{pa}\widetilde{\nabla}_j \widetilde{g}_{qb} + 2\widetilde{\nabla}_a \widetilde{g}_{jp} \widetilde{\nabla}_{q}\widetilde{g}_{ib} 
   -2\widetilde{\nabla}_a \widetilde{g}_{jp}\widetilde{\nabla}_b \widetilde{g}_{iq} - 2\widetilde{\nabla}_j \widetilde{g}_{pa} \widetilde{\nabla}_b \widetilde{g}_{iq} -2\widetilde{\nabla}_i \widetilde{g}_{pa} \widetilde{\nabla}_b \widetilde{g}_{jq}\right), \\
   &= \frac{\pa g_{ij}}{\pa t}(0)= -\Ric(\widetilde{g})_{ij}-\frac{n+1}2 \widetilde{g}_{ij}.
\end{aligned}  
\label{ahc.10}\end{equation}
Comme $\widetilde{g}=g_0$ est une métrique $\ahc$, on voit donc par le Lemme~\ref{ahc.3} qu'il existe $\delta>0$ tel que 
\begin{equation}
 f\in \rho^{\delta}\CI_{\Theta}([0,T)\times X\setminus \pa X; T^0_2 (X\setminus \pa X)).
 \label{ahc.10b}\end{equation}
On peut alors voir \eqref{ahc.9} comme une équation parabolique linéaire,
\begin{equation}
  \frac{\pa u}{\pa t}= L_t  u + f, \quad u(0)=0,
\label{ahc.11}\end{equation}
où $t\to L_t \in\Diff^2_{\Theta}(X\setminus \pa X ; \End(T^0_2(X\setminus \pa X)))$ est la famille d'opérateurs uniformément $\Theta$-elliptique donnée par
\begin{equation}
   (L_t v)_{ij}= \frac{g^{ab}}{2} \widetilde{\nabla}_a \widetilde{\nabla}_b v_{ij} + (Q_1(g_0,u, \widetilde{\nabla}u)\cdot  \widetilde{\nabla}v)_{ij} + (Q_2(g_0,u, \widetilde{\nabla}u)\cdot v)_{ij}.
\label{ahc.11b}\end{equation}
Par \eqref{ahc.10b} et en appliquant le Corollaire~\ref{par.3}, on voit donc que 
\begin{equation}
 u\in \rho^{\delta}\CI_{\Theta}(X\setminus \pa X; T^0_2(X\setminus \pa X)).
\label{ahc.12a}\end{equation}
\end{proof}

Lorsque la métrique $g_0$ est polyhomogène, il est naturel de se demander s'il en sera de même pour la solution du flot de Ricci-DeTurck.  Le théorème qui suit répond par l'affirmative.
\begin{theoreme}
Soit $g_0$ une métrique asymptotiquement hyperbolique complexe sur l'intérieur d'une variété compacte à bord $X$.  Si $g_0$ est polyhomogène avec ensemble indiciel positif $E$ et si $g(t)$ pour $t\in [0,T)$ est la solution du flot de Ricci-DeTurck  \eqref{ahc.5} avec métrique initiale $g_0$, alors 
$$
           g\in \cA^{E_{\infty}}_{\phg}([0,T)\times X; {}^{\Theta}T^* X \otimes {}^{\Theta}T^* X).            
$$
\label{ahc.7}\end{theoreme}

Lorsque la métrique $g_0$ est lisse jusqu'au bord, c'est-à-dire que $E= \bbN_0\times \{ 0\}$, on peut facilement déduire de ce théorème le résultat suivant.
\begin{corollaire}
Soit $g_0$ une $\Theta$-métrique lisse jusqu'au bord sur l'intérieur d'une variété compacte à bord $X$.  Alors la solution $g(t)$ pour $t\in[0,T)$ au flot de Ricci \eqref{ahc.4} avec métrique initiale $g_0$ est telle que
$$
       g\in \CI([0,T)\times X; {}^{\Theta}T^* X \otimes {}^{\Theta}T^* X).
$$ 
\label{ahc.8}\end{corollaire} 
\begin{proof}
Dans ce cas, on sait que la solution $\hat{g}(t)$ au flot de Ricci-Deturck est lisse jusqu'au bord.  Cela veut dire que le champ de vecteur $W$ de \eqref{ahc.4a} est lisse jusqu'au bord, en fait que $W\in \CI([0,T)\times X; TX)$ et que $\left.W\right|_{\pa X}=0$.  C'est donc dire que la famille de difféomorphismes $\varphi_t$ engendrée par $W$ est constituée de difféomorphismes de $X$ qui figent le bord $\pa X$.  La solution $g(t)= (\varphi_t^{-1})^*\hat{g}(t)$ au flot de Ricci sera donc aussi lisse jusqu'au bord.     
\end{proof}

\begin{proof}[Démonstration du Théorème~\ref{ahc.7}]

Comme dans la preuve du Lemme~\ref{ahc.6a}, posons $u=g-g_0$.  Par le Lemme~\ref{ahc.6a} et \eqref{ahc.12a}, on sait donc que   
\begin{equation}
 u\in \rho^{\delta}\CI_{\Theta}(X\setminus \pa X; T^0_2(X\setminus \pa X)), \quad \forall \; \delta<p,
\label{ahc.12}\end{equation}
où $p$ est le plus petit réel positif tel que $(p,0)\in E_{\infty}$.  C'est le premier pas d'une preuve par induction du Théorème.  
Plus précisément, soit $\{ z_j \}_{j\in\bbN_0}$ une suite strictement croissante de nombres réels telle que 
\begin{gather}
        z_0=0 \; \mbox{et} \;  (z_j,0)\in E_{\infty} \;\forall j\in \bbN, \\
         (z,k)\in E_{\infty}\; \Longrightarrow \; z=z_j \; \mbox{pour un certain} \; j\in\bbN_0.  
\end{gather}
Par définition de la suite $\{z_j\}$ et de $E_{\infty}$, remarquons que 
\begin{equation}
            z_j\ge z_1 \ge z_{j+1}-z_j  \quad \forall \; j\in \bbN.
\label{ahc.13}\end{equation}
Pour établir que $u$, et donc $g$, est polyhomogène avec ensemble indiciel $E_{\infty}$, il faut que pour tout $j\in\bbN_0$, on puisse trouver $u_j\in \cA^{E_{\infty}}_{\phg}([0,T)\times X; {}^{\Theta}T^* X\otimes {}^{\Theta}T^* X)$ tel que 
\begin{equation}
  u-u_j \in \rho^{\nu}\CI_{\Theta}([0,T)\times X\setminus \pa X; T^0_2(X\setminus \pa X)), \quad \forall \; \nu<z_{j+1}.
\label{ahc.14}\end{equation}
Par \eqref{ahc.12}, pour $j=0$, on peut prendre $u_0\equiv 0$.  Procédons par récurrence et supposons maintenant que pour un certain $j\in \bbN$, on ait déjà trouvé $u_{j-1}\in\cA^{E_{\infty}}_{\phg}([0,T)\times X; {}^{\Theta}T^* X\otimes {}^{\Theta}T^* X)$ tel que  
\begin{equation}
  u-u_{j-1} \in \rho^{\nu}\CI_{\Theta}([0,T)\times X\setminus \pa X; T^0_2(X\setminus \pa X)) \quad \forall \; \nu<z_{j}.
\label{ahc.15}\end{equation}
Il faut alors trouver $u_j$ pour que \eqref{ahc.14} soit aussi valide.  Pour ce faire, posons $v= u-u_{j-1}$.  Comme l'équation d'évolution de $u$ est donnée par \eqref{ahc.11} avec $L_t$ défini par \eqref{ahc.11b}, on a que
\begin{equation}
  \frac{\pa v}{\pa t}= L_tv +f^j, \quad f^j:= f+L_t u_{j-1} - \frac{\pa u_{j-1}}{\pa t}.
\label{ahc.16}\end{equation}
Grâce à \eqref{ahc.15} et la définition de $f^j$, on voit que $f^j\in \rho^{\nu}\CI_{\Theta}([0,T)\times X\setminus \pa X; T^0_2(X\setminus \pa X))$ pour tout $\nu<z_j$.  Cependant, comme la famille d'opérateurs $L_t$ dépend de $v$ et n'est donc pas a priori polyhomogène, on ne peut pas conclure que $f^j$ est polyhomogène.  À nouveau, grâce à \eqref{ahc.15}, on sait toutefois que 
\begin{equation}
L_t\in \Diff^2_{\Theta, E_{\infty}}([0,T)\times X; {}^{\Theta}T^*X\otimes{}^{\Theta}T^*X)+ \rho^{\nu}\Diff^2_{\Theta}(X\setminus\pa X; T^0_2(X\setminus\pa X))
\label{ahc.17}\end{equation}  
pour tout $\nu<z_j$.  Comme $u_{j-1}\in \rho^{\nu}\CI(X\setminus \pa X; T^0_2(X\setminus \pa X))$ pour tout $\nu<z_1$ et que par \eqref{ahc.13}  $z_1\ge z_{j+1}-z_j$, on peut donc conclure que
\begin{equation}
  f^j\in \cA^{E_{\infty}}_{\phg}([0,T)\times X; {}^{\Theta}T^* X\otimes {}^{\Theta}T^* X)+  \rho^{\nu}\CI_{\Theta}([0,T)\times X\setminus \pa X; T^0_2(X\setminus \pa X)) 
  \label{ahc.18}\end{equation}
pour tout $\nu<z_{j+1}$.  Si $v$ est polyhomogène avec ensemble indiciel $E_{\infty}$, alors pour $k\in\bbN_0$ le plus grand entier tel que $(z_j,k)\in E_{\infty}$, on peut restreindre l'équation \eqref{ahc.16} au coefficient d'ordre $\rho^{z_j}(\log\rho)^k$ sur le bord pour obtenir une équation différentielle ordinaire linéaire
\begin{equation}
  \frac{\pa v_{\pa}}{\pa t}= A\cdot v_\pa+ f^j_{\pa}, \quad v_{\pa}(0)=0.
\label{ahc.19}\end{equation} 
Insistons sur le fait ici que le coefficient $A$ ne dépend pas de $v_{\pa}$, et donc que \eqref{ahc.19} est bien une équation linéaire et non une équation non-linéaire déguisée en équation linéaire comme \eqref{ahc.16}.  La solution $v_{\pa}$ de cette équation donne par construction la restriction de $v$ au bord à l'ordre $\rho^{z_j}(\log \rho)^k$.  Maintenant, même si on ne sait pas a priori si $v$ est polyhomogène, remarquons que l'équation \eqref{ahc.19} a malgré tout un sens et que sa solution $v_{\pa}$ donne le candidat naturel de ce que serait la restriction à l'ordre $\rho^{z_j}(\log \rho)^k$ de $v$.  Pour montrer que c'est bien le cas, posons
$$
           w=v- w_{\pa}, \quad w_{\pa}:=\rho^{z_j}(\log\rho)^k \Xi(v_{\pa}),
$$ 
où 
$$
\Xi: \CI(\pa X;{}^{\Theta}T^* X\otimes {}^{\Theta}T^* X)\to \CI(X; {}^{\Theta}T^* X\otimes {}^{\Theta}T^* X)
$$ 
est un choix d'application d'extension,  c'est-à-dire que $\left.\Xi(h)\right|_{\pa X}=h$ pour tout $h\in  \CI(\pa X;{}^{\Theta}T^* X\otimes {}^{\Theta}T^* X)$.  On a alors que
$$
  \frac{\pa w}{\pa t}= L_t w + h^j, \quad w(0,\cdot)=0, \quad h^j:= f^j+ L_t w_{\pa} -\frac{\pa w_{\pa}}{\pa t}.
$$ 
Grâce à \eqref{ahc.18} et notre choix de $w_{\pa}$, on aura que
$$
              h^j\in \cA^{E_{\infty}}_{\phg}(X;   {}^{\Theta}T^* X\otimes {}^{\Theta}T^* X)+ \rho^{\nu}\CI_{\Theta}(X\setminus \pa X; T^0_2(X\setminus \pa X)),\quad \forall \;\nu<z_{j+1}
$$
et que
$$
   h^j\in \rho^{z_j}(\log \rho)^{k-1}\CI_{\Theta}(X\setminus \pa X; T^0_2(X\setminus \pa X)).
$$
 En répétant cet argument $k$ fois, on peut donc trouver $b_k, b_{k-1},\ldots, b_0\in \CI([0,T)\times \pa X; {}{\Theta}T^* X\otimes {}^{\Theta}T^* X)$ tel que
 $$
     \hat{w}=v - \sum_{\ell=0}^k \Xi(b_{\ell}) \rho^{z_j}(\log \rho)^{\ell}
 $$ 
a pour équation d'évolution
$$
   \frac{\pa\hat{w}}{\pa t}= L_t \hat{w} + \hat{f}^j, \quad \hat{w}(0,\cdot)=0, 
$$
avec 
$$
  \hat{f}^j\in  \rho^{\nu}\CI_{\Theta}(X\setminus \pa X; T^0_2(X\setminus \pa X)),\quad \forall \;\nu<z_{j+1}.
$$
En appliquant le Corollaire~\ref{par.3}, on conclut donc que $\hat{w}\in \rho^{\nu}\CI_{\Theta}(X\setminus \pa X; T^0_2(X\setminus \pa X))$ pour tout $\nu<z_{j+1}$.  Le pas d'induction est donc complété en prenant
$$
     u_j= u_{j-1} + \sum_{\ell=0}^k \Xi(b_{\ell}) \rho^{z_j}(\log \rho)^{\ell}.
$$ 

\end{proof}

\section{Polyhomogénéité des métriques kählériennes $\ahc$ le long du flot de Ricci} \label{kahc.0}

Soit $\cU$ une variété compacte complexe à bord, c'est-à-dire que $\cU$ est une variété compacte à bord plongée dans une variété complexe.  Un exemple à garder à l'esprit est le cas plus particulier où $\cU$ est un domaine dans  $\bbC^n$.    Soit $r\in \CI(\cU)$ un choix de fonction de définition du bord.  Dénotons par $\eta\in \CI(\pa\cU; T^*\pa \cU)$ la restriction au bord de la 1-forme $i\pa r-i\db r$.  Le noyau de $\eta$ définit alors une distribution $H$ de codimension 1 sur $\pa \cU$.  La structure complexe $J$ de $\cU$ se restreint naturellement à une structure complexe sur $H$, ce qui confère à $\pa \cU$ une structure de variété $\CR$.  On supposera que le bord de $\cU$ est strictement pseudoconvexe, ce qui signifie que la forme $-i\pa\db r(\cdot, J\cdot)= d\eta(\cdot, J\cdot)$ est définie positivement sur $H$.  En particulier, $(X,H)$ est une variété de contact ayant $\eta$ comme forme de contact.     

Dans ce contexte, on a une $\Theta$-structure associée.  Comme expliqué dans \cite{EMM}, celle-ci n'est toutefois pas définie directement sur $\cU$.  Il faut au préalable introduire la racine carrée $\cU_{\frac12}$ de la variété à bord $\cU$.  La racine carrée $\cU_{\frac12}$ est la variété obtenue à partir de la variété $\cU$ en ajoutant $r^{\frac12}$ à l'anneau $\CI(\cU)$ des fonctions lisses sur $\cU$, de sorte que $\CI(\cU_{\frac12})$ est engendré par $r^{\frac12}$ et $\CI(\cU)$.  Un choix de fonction de définition du bord sur $\cU_{\frac12}$ est donc donné par $\rho= r^{\frac12}$.  Pour alléger la notation, on dénotera la racine carrée $\cU_{\frac12}$ de $\cU$ par $X$.  Remarquons que l'inclusion de $\CI(\cU)$ dans $\CI(X)$ donne lieu à une application canonique
$$
          \iota_{\frac12} : X\to \cU.
$$
Cette application n'est pas un difféomorphisme, mais sa restriction à $\pa X$ ou à $X\setminus \pa X$ donnent des difféomorphismes canoniques entre $\pa X$ et $\pa \cU$ d'une part et entre $X\setminus \pa X$ et $\cU\setminus \pa \cU$ d'autre part.  La structure complexe sur $\cU$ permet de définir une $\Theta$-structure sur $X$ en posant
$$
                   \Theta=   \left. \iota_{\frac12}^*( -i\db r+ i\pa r)\right|_{\pa X} \in \CI(\pa X; \left.TX\right|_{\pa X}).
$$  
 La forme 
\begin{equation}
\overline{\Theta}= \iota_{\frac12}^*\left( -i\db r + i\pa r\right) \in \CI(X; TX)
\label{kahc.1a}\end{equation}
est une extension naturelle de $\Theta$ à $X$. Évidemment, la forme $\Theta$ dépend du choix de la fonction de définition du bord $r$.  Cependant, on vérifie aisément que la classe conforme de $\Theta$ est indépendante du choix de $r$.  Dans ce cadre particulier, une métrique $\ahc$ $g$ sur $X\setminus \pa X$ est une $\Theta$-métrique complète sur $X\setminus \pa X$ pour laquelle il existe une forme $\eta\in \CI(X;TX)$ avec $\left. \eta\right|_{\pa X}$ dans la classe conforme de $\Theta$ et $\delta>0$ telle que 
\begin{equation}
         g- \left( \frac{4d\rho^2}{\rho^2}+ \frac{\eta^2}{\rho^2}+ \frac{d\eta(\cdot,J\cdot)}{\rho}  \right) \in \rho^\delta\CI_{\Theta}(X\setminus \pa X; T^0_2(X\setminus\pa X)),
\label{kahc.3c}\end{equation}  
où $\rho=r^{\frac12}$.  On dira que la métrique $g$ est lisse jusqu'au bord ou polyhomogène si elle est lisse jusqu'au bord ou polyhomogène en tant que $\Theta$-métrique.  
\begin{remarque}
Cette notion de métrique $\ahc$ ne dépend pas du choix de la fonction de définition du bord $r$.  En effet, soit $\widetilde{r}= e^f r$ un autre choix, où $f\in \CI(U)$.  Alors la condition \eqref{kahc.3c} équivaut à demander que pour un certain $\delta>0$,  
$$
g- \left( \frac{4d\widetilde{\rho}^2}{\widetilde{\rho}^2}+ \frac{\widetilde{\eta}^2}{\widetilde{\rho}^4}+ \frac{d\widetilde{\eta}(\cdot,J\cdot)}{\widetilde{\rho}^2}  \right) \in \widetilde{\rho}^\delta\CI_{\Theta}(X\setminus \pa X; T^0_2(X\setminus\pa X))
$$
avec $\widetilde{\rho}= e^{\frac{f}{2}}\rho$ et $\widetilde{\eta}= e^{f}\eta$.
\label{kahc.3b}\end{remarque}
Autrement dit, une métrique asymptotiquement hyperbolique complexe est une $\Theta$-métrique dont le comportement près de $\pa X$ est modelé sur celui de la métrique 
$$
\frac{4d\rho^2}{\rho^2}+ \frac{\eta^2}{\rho^4}+ \frac{d\eta(\cdot,J\cdot)}{\rho^2} .
$$

Lorsque $\cU$ est une variété kählérienne à bord, Il est assez facile de construire des exemples de métriques kählériennes $\ahc$ sur $X\setminus \pa X$.  En effet, soit $\omega$ une forme kählérienne sur $\cU$.  Elle donne par restriction une forme kählérienne sur $X\setminus \pa X= \cU\setminus \pa \cU$.  La métrique kählérienne associée $g_{\omega}$ est bien entendue incomplète.  Pour la rendre complète, on peut lui ajouter la contribution venant d'un potentiel 
$$
      u= -4 \log r= -8\log \rho,
$$
où $\rho=r^{\frac12}$ est un choix de fonction de définition du bord de $X$.  En choisissant une constante $c>0$ suffisamment grande, on obtient alors une 2-forme non-dégénérée
\begin{equation}
 \omega_c = c\omega +\frac{i}{2} \pa \db u = c\omega +2i\left(  - \frac{\pa \db r}{r} +\frac{\pa r \wedge \db r}{r^2} \right).
\label{kahc.1}\end{equation}
En effet, l'hypothèse de stricte pseudoconvexité du bord de $\cU$ nous assure que la forme $i\pa \db u$ est non-dégénérée dans un voisinage de $\pa X$.  En prenant $c>0$ assez grand, la forme $\omega_c$ sera donc kählérienne partout sur $X\setminus \pa X$.  La métrique associée $g_{\omega_c}$ est alors clairement complète sur $X\setminus \pa X$.  C'est en fait une $\Theta$-métrique par rapport à la $\Theta$-structure associée à la structure complexe $\cU$.  Pour le voir, remarquons qu'un calcul simple montre que 
\begin{equation}  
  g_{\omega_c}= c g_{\omega} + \frac{\gamma}{\rho^2} + \frac{4d\rho^2}{\rho^2} + \frac{ \overline{\Theta}^2}{\rho^4}, \quad \mbox{avec} \quad \gamma= \iota_{\frac12}^*\left(  - 2i \pa \db r (\cdot, J\cdot) \right),
\label{kahc.2}\end{equation}
est clairement une $\Theta$-métrique lisse jusqu'au bord.  La métrique $g_{\omega_c}$ est un exemple typique d'une métrique kählérienne asymptotiquement hyperbolique complexe.  Comme le lemme qui suit l'indique, le fait que $\eta= \overline{\Theta}$ dans cet exemple n'est pas un hasard.

\begin{lemme}
Soit $(\cU,J)$ une variété kählérienne compacte à bord strictement pseudoconvexe.  Soit $r\in \CI(\cU)$ un choix de fonction de définition du bord et $\Theta= \left.\iota^*_{\frac12}(i\pa r-i\db r) \right|_{\pa X}$ la $\Theta$-structure associée.  Si $g$ est une métrique kählérienne asymptotiquement hyperbolique complexe, alors il existe $\delta>0$ tel que 
$$
g- \left( \frac{4d\rho^2}{\rho^2}+ \frac{\overline{\Theta}^2}{\rho^4}+ \frac{d\overline{\Theta}(\cdot,J\cdot)}{\rho^2}  \right) \in \rho^\delta\CI_{\Theta}(X\setminus \pa X; T^0_2(X\setminus\pa X)),
$$  
où $\rho= r^{\frac12}$ et $\overline{\Theta}$ est donné par \eqref{kahc.1a}. 
\label{kahc.4}\end{lemme}  
\begin{proof}
Par hypothèse, il existe une $1$-forme $\eta\in \CI(X;T^*X)$ avec $\left. \eta \right|_{\pa X} = e^{f}\Theta$ pour une certaine fonction $f\in \CI(\pa X)$ et $\delta>0$ tels que \eqref{kahc.3c} est valable.  Sans perte de généralité, on peut supposer que $f$ est la restriction d'une fonction lisse sur $\cU$ (et donc sur $X$) qu'on dénotera aussi $f$.  En termes de la forme $\overline{\Theta}$ et en prenant $\delta>0$ plus petit si nécessaire, en particulier plus petit que $1$, l'équation \eqref{kahc.3c} équivaut à
\begin{equation}
   g- \left( \frac{4d\rho^2}{\rho^2}+ \frac{e^{2f}\overline{\Theta}^2}{\rho^4}+ \frac{e^f d\overline{\Theta}(\cdot,J\cdot)}{\rho^2}  \right) \in \rho^\delta\CI_{\Theta}(X\setminus \pa X; T^0_2(X\setminus\pa X)).
\label{kahc.5}\end{equation}
Il faut montrer que $\left. f\right|_{\pa X}\equiv 1$, car on pourra alors choisir l'extension de $f$ à $X$ de sorte que $f\equiv 1$ partout sur $X$.  D'abord, la 2-forme associée à la métrique 
$$
     h=\frac{4d\rho^2}{\rho^2}+ \frac{e^{2f}\overline{\Theta}^2}{\rho^4}+ \frac{e^f d\overline{\Theta}(\cdot,J\cdot)}{\rho^2}= \frac{dr^2}{r^2}+ \frac{e^{2f}\overline{\Theta}^2}{r^2}+ \frac{e^f d\overline{\Theta}(\cdot,J\cdot)}{r}$$
est donnée par
\begin{equation}
\begin{aligned}
  \omega_h &= h(J\cdot, \cdot) = \frac{1+e^{2f}}{2r^2} \overline{\Theta} \wedge dr + -2ie^f\frac{\pa\db r}{r} \\
  &= i(1+e^{2f}) \frac{\pa r \wedge \db r}{r^2} -2ie^f \frac{\pa\db r}{r} \\
  &=  2ie^f\left( \frac{\pa r \wedge \db r}{r^2} - \frac{\pa \db r}{r}  \right)  + i(1-e^f)^2\frac{\pa r \wedge \db r}{r^2}.
\end{aligned}
\end{equation}
On a donc que 
\begin{equation}
\begin{aligned}
  d\omega_h &=  2ie^f \frac{df\wedge \pa r\wedge \db r}{r^2} -2ie^f \frac{df\wedge \pa \db r}{r} -2i(1-e^f)e^f \frac{df\wedge \pa r\wedge \db r}{r^2} +  \\
  &  \hspace{1cm} + i(1-e^f)^2 \left( \frac{\db\pa r\wedge \db r - \pa r \wedge \pa \db r}{r^2} \right)         \\
  &=  2ie^{2f} \frac{df\wedge \pa r \wedge \db r}{r^2} - 2ie^f \frac{df\wedge\pa\db r}{r} -i(1-e^f)^2 \frac{dr\wedge \pa\db r}{r^2}.
\end{aligned}  
\label{kahc.6}\end{equation}
Puisque la forme de Kähler de $g$ est fermée, on déduit de \eqref{kahc.5} que $d\omega_h\in \rho^{\delta}\CI_{\Theta}(X\setminus \pa X; \Lambda^3(T^*(X\setminus \pa X))$ pour un certain $0<\delta\le 1$.  Puisque 
$$
df\in\rho\CI_{\Theta}(X\setminus \pa X; T^*(X\setminus \pa X)),
$$ 
des trois termes dans la dernière ligne de \eqref{kahc.6}, seul le dernier n'est potentiellement pas contenu dans $\rho^{\delta}\CI_{\Theta}(X\setminus \pa X; \Lambda^3(T^*(X\setminus \pa X)))$, mais dans $\CI_{\Theta}(X\setminus \pa X; \Lambda^3(T^*(X\setminus \pa X)))$.  Il faut donc que $\left. f\right|_{\pa X}\equiv 1$ pour que la forme $d\omega_h$ soit asymptotiquement nulle.

\end{proof}

 Lorsque la métrique initiale $g_0$ du flot de Ricci normalisé \eqref{ahc.4} est une métrique kählérienne $\ahc$, on peut tirer avantage de la structure kählérienne pour étudier les questions de convergence et d'existence pour tout temps positif.    En fait, en prenant l'ansatz $\widetilde{\omega}_t= \omega_t + i\pa \db u$ avec
$$
     \omega_t= -\frac{2}{n+1}Ric(\omega_0)+ e^{-\frac{n+1}{2}t}(\omega_0+\frac{2}{n+1}\Ric(\omega_0)),
$$
on peut réécrire le flot de Ricci normalisée \eqref{ahc.4} en termes du potentiel $u$ de la façon suivante,
\begin{equation}
  \frac{\pa u}{\pa t}= \log\left( \frac{(\omega_t+ i\pa \db u)^n}{\omega_0^n} \right) -\frac{n+1}{2}u, \quad u(0,\cdot)=0.
\label{ricci.2}\end{equation}
En considérant cette équation pour des métriques kählériennes complètes de courbure sectionnelle bornée, Chau \cite{Chau2004} a obtenu un critère de nature cohomologique pour la convergence du flot vers une métrique de Kähler-Einstein.  Pour notre métrique kählérienne $g_0$ asymptotiquement hyperbolique complexe, le critère de Chau stipule que le flot de Ricci ayant $g_0$ pour métrique de départ convergera vers une métrique de Kähler-Einstein pourvu qu'il existe une fonction lisse bornée $f$ sur $X\setminus \pa X$ telle que
$$
                            Ric(\omega_0)+\frac{n+1}{2}\omega_0 = i\pa\db f.
$$ 
\begin{exemple}
Lorsque $\cU$ est un domaine strictement convexe dans $\bbC^n$ et $g_0$ est la métrique ayant pour forme kählérienne
$$
        \omega_0= -2i\pa\db \log r
$$ 
où $r$ est un choix de fonction de définition du bord qui soit en même temps strictement plurisuperharmonique, on peut appliquer le critère avec 
$$
f=-(n+1)\log r -\log \det((g_0)_{p\overline{q}})
$$
pour voir que le flot converge vers une métrique de Kähler-Einstein, plus précisément la métrique de Cheng-Yau \cite{Cheng-Yau}.
\label{ricci.3}\end{exemple}
Cet exemple admet la généralisation suivante.
\begin{exemple}
Soit $v\in\Omega^{2n}_{\Theta}=\CI_{\Theta}(X;\Lambda^{2n}({}^{\Theta}T^*X))$ un choix de $\Theta$-forme volume lisse jusqu'au bord.  Cette forme volume définit sur $X\setminus \pa X$ une métrique hermitienne pour le fibré canonique $K_{\cU}$.  En coordonnées locales complexes sur $\cU$, la courbure associée est donnée par
$$
                \Omega= \db\pa \log u,  \quad v= u dz_1\wedge d\overline{z}_1\wedge \ldots \wedge dz_n\wedge d\overline{z}_n$$ 
et on vérifie aisément que $\Omega\in \Omega^2_{\Theta}(X)$.  Supposons qu'on puisse choisir $v$ de sorte que $\omega_0= \frac{2i\Omega}{n+1}$ soit une forme kählérienne, voir \cite{Cheng-Yau, Mok-Yau, Coevering2012} pour des situations où un tel choix est possible.  On peut alors appliquer le critère de Chau pour le flot de Ricci débutant avec $\omega_0$ en prenant 
$$
         f= \log \left( \frac{v}{\omega_0^n} \right)
$$    
pour voir que le flot converge vers une métrique de Kähler-Einstein.            
\label{ricci.4}\end{exemple}

Dans cette section, on s'intéresse au comportement à l'infini de $\Theta$-métriques kählériennes $\ahc$ polyhomogènes lorsqu'elles sont modifiées par le flot de Ricci normalisé.

\begin{theoreme}
Soit $g_0$ une $\Theta$-métrique kählérienne $\ahc$  qui soit polyhomogène avec ensemble indiciel positif $E$.  Si $u\in \CI_{\Theta}([0,T]\times X\setminus \pa X)$ est une solution de \eqref{ricci.2} jusqu'au temps $T>0$, alors en fait $u\in\cA^{E_{\infty}}_{\phg}([0,T]\times X)$.  En particulier, la solution $\widetilde{g}_t$ sera polyhomogène avec ensemble indiciel $E_{\infty}$ pour tout $t\in[0,T]$.  
\label{ricci.5}\end{theoreme}
En prenant $E=\bbN_0\times \{0\}$, on obtient le corollaire suivant.
\begin{corollaire}
Soit $g_0$ une métrique kählérienne $\ahc$ qui soit lisse jusqu'au bord.  Si $u\in \CI_{\Theta}([0,T]\times X\setminus \pa X)$ est une solution de \eqref{ricci.2} jusqu'au temps $T>0$, alors en fait $u\in\CI([0,T]\times X)$.  En particulier, la solution $\widetilde{g}_t$ sera une $\Theta$-métrique lisse jusqu'au bord pour tout $t\in[0,T]$.  
\label{ricci.6}\end{corollaire}

D'abord, supposons que $u$ soit bien polyhomogène.  Dans ce cas, on peut restreindre l'équation \eqref{ricci.2} au bord $\pa X$.  Par le Lemme~\ref{ahc.3} et le fait que $\left.\pa\db u\right|_{\pa X}=0$ en tant qu'élément de $\cA^{E_{\infty}}_{\phg}([0,T]\times X; \Lambda^2({}^{\Theta}T^* X))$, on voit que la restriction de \eqref{ricci.2} au bord $\pa X$ est 
\begin{equation}
\frac{\pa u_{\pa}}{\pa t}= -\frac{n+1}{2} u_{\pa}, \quad u_{\pa}(0)=0, \quad u_{\pa}:= \left. u \right|_{\pa X}.
\label{ricci.7}\end{equation}
La solution de cette équation est $u_{\pa}(t,\cdot)\equiv0$ pour tout $t$.  Maintenant, même si on ne sait pas a priori si la restriction de $u$ sur $\pa X$ a un sens, l'équation \eqref{ricci.7} elle en a un et sa solution $u_\pa\equiv 0$ donne le candidat naturel de ce que serait la restriction de $u$ si elle existait.  Pour vérifier que $u_{\pa}$ est effectivement la restriction de $u$ au bord, on utilisera essentiellement le fait que la métrique initiale est polyhomogène et le Corollaire~\ref{par.3}.  Initialement toutefois, on ne pourra obtenir un tel résultat que sur un petit intervalle de temps.  Comme on pourra trouver une borne inférieure uniforme à la longueur de cet intervalle en termes de la famille de métriques $g(t)$, $t\in [0,T]$, on pourra toutefois appliquer ce stratagème un nombre fini de fois pour couvrir l'intervalle $[0,T]$ en entier.  Pour ce faire, on considérera pour un choix de $t_0\in [0,T)$ une version translatée de l'équation \eqref{ricci.2}, à savoir
\begin{equation}
 \frac{\pa \hu}{\pa t}= \log\left( \frac{(\omega_{t+t_0}+ i\pa \db u(t+t_0,\cdot))^n}{\omega_0^n} \right) -\frac{n+1}{2}u(t+t_0,\cdot), \quad \hu(0,\cdot)=0,
 \label{ricci.8}\end{equation} 
où $\hu(t,\cdot)= u(t+t_0,\cdot)-u(t_0,\cdot)$.  En posant $\homega_t= \omega_{t+t_0}+ i\pa \db u(t_0,\cdot)$, de sorte que $\homega_t+ i\pa\db \hu=\widetilde{\omega}_{t+t_0}$, cette équation peut être réécrite de la façon suivante,
\begin{equation}
 \frac{\pa \hu}{\pa t}= \log\left( \frac{(\homega_t+ i\pa \db \hu)^n}{\omega_0^n} \right) -\frac{n+1}{2}(\hu+ u(t_0,\cdot)), \quad \hu(0,\cdot)=0.
 \label{ricci.9}\end{equation}
On suppose alors que $u(t_0,\cdot)\in \cA^{E_\infty}_{\phg}(X)$ avec $u_{\pa}(t_0,\cdot)= \left. u(t_0,\cdot)\right|_{\pa X}=0$.  Comme précédemment, même si on ne sait pas a priori si la restriction de $\hu$ au bord a un sens, on sait que l'équation \eqref{ricci.9} a quant à elle une restriction bien définie sur $\pa X$,
$$
       \frac{\pa \hu_{\pa}}{\pa t}= -\frac{n+1}{2} \hu_{\pa}, \quad \left. \hu\right|_{t=0}=0.
$$
La solution de cette équation est donnée par $\hu_{\pa}\equiv 0$.  On s'attend donc à ce que la restriction de $\hu$ au bord soit nulle.  Pour voir que c'est le cas pour $t>0$ assez petit, remarquons dans un premier temps que $\homega_t$ n'est pas nécessairement une $2$-forme non-dégénérée, mais qu'elle le sera pour $t>0$ assez petit du fait que $\homega_0=\widetilde{\omega}_{t_0}$ est non-dégénérée par hypothèse.  Pour avoir une estimation uniforme et indépendante de notre choix de $t_0$, notons qu'ils existe une constante $C>0$ telle que 
\begin{equation}
  \frac{\omega_0}{C} \le \widetilde{\omega}_t \le C\omega_0, \quad  \forall t\in [0,T].
\label{ricci.10}\end{equation} 
En prenant la constant $C>0$ plus grande si nécessaire, on peut aussi supposer que 
\begin{equation}
       \left\| \frac{\pa}{\pa t} \pa \db u\right\|_{\cC^0_{\Theta}(X\setminus\pa X; \Lambda^2(T^*(X\setminus\pa X)))} < C, \quad \forall \; t\in [0,T].
\label{ricci.11}\end{equation}
Comme d'autre part 
$$
        \frac{\pa \homega_t}{\pa t}= -\frac{n+1}{2} e^{-\frac{n+1}{2}(t+t_0)}\left(\frac{2}{n+1}\Ric(\omega_0)+\omega_0\right),
$$
on voit donc qu'on peut choisir $\tau>0$ indépendamment de notre choix de $t_0$ de sorte que 
\begin{equation}
   \homega_t> \frac{\omega_0}{2C}, \quad \forall\; t\in [0,\tau_{t_0}], \quad \mbox{où} \; \tau_{t_0}:= \min\{\tau, T-t_0\}.
\label{ricci.12}\end{equation}
La $2$-forme $\homega_t$ est donc kählérienne pour $t\in[0,\tau_{t_0}]$, ce qui nous permet de réécrire le terme logarithmique de \eqref{ricci.9} comme suit,
\begin{equation}
\begin{aligned}
\log\left(  \frac{(\homega_t+i\pa \db u)^n}{\omega_0^n}  \right)&= \log\left(  \frac{(\homega_t+i\pa \db u)^n}{\homega_t^n}  \right)+ \log\left(\frac{\homega_t^n}{\omega_0^n}\right) \\
                    &= \log\left( 1+ \hDelta_t\hu +F(\pa\db\hu,\homega_t) \right) + \log\left(\frac{\homega_t^n}{\omega_0^n}\right),
\end{aligned}
\label{ricci.13}\end{equation}
où $\hDelta_t$ est le $\db$-laplacien de $\homega_t$ et la fonction $F$ est un polynôme de degré $n$ dans la première variable sans terme constant ou linéaire.  On a donc que
$$
\log\left( 1+ \hDelta_t\hu +F(\pa\db\hu,\homega_t) \right)= \hDelta_t \hu + H(\pa\db \hu, \homega_t)\pa\db\hu,
$$ 
où $H$ est une application à valeur dans $\CI_{\Theta}(X\setminus \pa X;\Lambda^2(T(X\setminus\pa X))$ s'annulant au moins linéairement dans la première variable lorsque celle-ci est zéro.  L'application $H$ est indépendante de notre choix de $t_0$.  Par \eqref{ricci.11} et \eqref{ricci.12}, cela signifie que qu'en prenant $\tau>0$ plus petit si nécessaire, mais toujours indépendant du choix de $t_0\in[0,T]$, on peut supposer que la famille lisse d'opérateurs $t\to L_t\in \Diff^2_{\Theta}(X\setminus \pa X)$ donnée par
$$
              L_t f= \hDelta_t f + H(\pa\db\hu, \homega_t)\pa\db f 
$$ 
est telle que $\pa_t- L_t$ est uniformément $\Theta$-parabolique pour $t\in[0,\tau_{t_0}]$, où $\tau_{t_0}= \min\{\tau, T-t_0\}$.  En termes de cette famille d'opérateurs, on peut réécrire \eqref{ricci.9} de la manière suivante,
\begin{equation}
\frac{\pa \hu}{\pa t}= \left(L_t \hu-\frac{n+1}{2}\hu\right) + \left(\log\left(\frac{\homega_t^n}{\omega_0^n}\right)-\frac{n+1}{2}u(t_0,\cdot)  \right), \quad \hu(0)=0, \quad t\in[0,\tau_{t_0}].
\label{ricci.14}\end{equation} 
Comme par hypothèse on a que 
$$
    f:=\log\left(\frac{\homega_t^n}{\omega_0^n}\right)-\frac{n+1}{2}u(t_0,\cdot) \in \cA^{E_{\infty}}_{\phg}([0,\tau_{t_0}]\times X), \quad \left. f\right|_{\pa X}=0,
$$
on peut appliquer le Corollaire~\ref{par.3} à l'équation \eqref{ricci.14} pour obtenir que
\begin{equation}
                      \hu\in \rho^{\delta}\CI_{\Theta}(X\setminus \pa X) \quad \forall \; 0<\delta<p,
\label{ricci.14a}\end{equation}
où $p>0$ est le plus petit réel positif tel que $(p,0)\in E_{\infty}$.  En procédant par récurrence, on peut plus généralement obtenir le résultat suivant.
\begin{proposition}
Soit $t_0\in [0,T)$.  Si $u(t_0,0)\in \cA^{E_\infty}_{\phg}(X)$ avec $\left. u(t_0,0)\right|_{\pa X}=0$, alors la solution $\hu$ de \eqref{ricci.9} est telle que
$$
                     \hu\in\cA^{E_{\infty}}_{\phg}([0,\tau_{t_0}]\times X), \quad \left.\hu\right|_{\pa X}=0.
$$
\label{ricci.15}\end{proposition}
\begin{proof}
Il suffit de procéder comme dans la preuve du Théorème~\ref{ahc.7} en considérant l'équation d'évolution \eqref{ricci.14} au lieu de \eqref{ahc.11}.
\end{proof}

La Proposition~\ref{ricci.15} nous permet alors de facilement compléter la démonstration du Théorème~\ref{ricci.5}.

\begin{proof}[Démonstration du Théorème~\ref{ricci.5}]

En appliquant la Proposition~\ref{ricci.15} pour $t_0=0$, on sait que le résultat du théorème sera valide sur l'intervalle $[0, \tau_{t_0}]$, où $\tau_{t_0}= \min\{\tau, T-t_0 \}$.  Si $\tau\ge T$, la démonstration est terminée.  Autrement, on peut appliquer à nouveau la Proposition~\ref{ricci.15} avec $t_0=\tau, \ldots k\tau$, où $k$ le plus grand entier tel que $k\tau < T$, pour couvrir tout l'intervalle $[0,T]$.

\end{proof}

\bibliography{Ricci_Theta}
\bibliographystyle{amsplain}
\hspace{0.5cm}

\small{Département de Mathématiques, Université du Québec à Montréal} 

 \small{\textit{Courriel}\! \!\!\!: rochon.frederic@uqam.ca}

\end{document}